\documentclass[a4paper, 11pt]{article} 
\usepackage[margin=2.3cm, includefoot, footskip=30pt]{geometry}

\usepackage[T1]{fontenc}

\usepackage{lmodern}

\usepackage{amssymb, mathtools, stmaryrd}
\usepackage{graphicx}

\usepackage{pgfplots}
\pgfplotsset{compat=1.11}
\usepackage{mathrsfs}
\usepackage{mathtools}
\usepackage{dsfont}
\usepackage{amsthm}
\usepackage{enumerate}
\usepackage{color}
\usepackage{verbatim}

\usepackage{array}

\usepackage{amssymb, amsopn, amscd}

\usepackage{microtype}

\usepackage{babel}

\usepackage{xcolor}
\usepackage[colorlinks]{hyperref}
\hypersetup{final}
\hypersetup{
    colorlinks,
    linkcolor={blue!40!black},
    citecolor={blue!40!black},
    urlcolor={blue!80!black}
}
\usepackage{nameref}
\usepackage[all]{hypcap}

\theoremstyle{plain}
\newtheorem{theorem}{Theorem}[section]
\newtheorem{corollary}[theorem]{Corollary}
\newtheorem{proposition}[theorem]{Proposition}
\newtheorem{lemma}[theorem]{Lemma}
\newtheorem{remark}[theorem]{Remark}
\newtheorem{definition}[theorem]{Definition}

\newtheorem{assumption}{Assumption}

\newtheorem{theoremI}{Theorem}

\newcommand{\backward}{\mathrm{Pend}}

\newcommand{\trace}{\mathrm{Tr}}
\newcommand{\Int}{\mathrm{Int}}
\newcommand{\clInt}{\overline{\mathrm{Int}}}
\newcommand{\Ext}{\mathrm{Ext}}
\newcommand{\clExt}{\overline{\mathrm{Ext}}}
\newcommand{\D}{\mathcal D}

\newcommand{\E}{\mathbb{E}}

\newcommand{\R}{\mathbb{R}}

\newcommand{\N}{\mathcal{N}}
\newcommand{\Q}{\mathbb{Q}}
\newcommand{\Z}{\mathbb{Z}}

\newcommand{\1}{\mathds 1}

\renewcommand{\P}{\mathbb{P}}

\numberwithin{equation}{section}

\begin{document}

\title{Connected components and topological ends of stationary planar forests}
\author{Tom Garcia-Sanchez\footnote{\href{mailto:tom.garcia-sanchez@imt-nord-europe.fr}{tom.garcia-sanchez@imt-nord-europe.fr} -- Univ. Lille, CNRS, IMT Nord Europe, UMR 8524 - Laboratoire Paul Painlevé, F-59000 Lille, France}}

\maketitle

\begin{abstract}
    We study the topological structure of random geometric forests $G$ in the Euclidean plane under mild assumptions: non-crossing edges, stationarity, and finite edge intensity. The framework covers a broad range of constructions, including models based on stationary point processes as well as lattices, and encompasses many already well-studied examples among drainage networks, geodesic forests arising from first- and last-passage percolation, and minimal or uniform spanning trees. First, denoting by $N_k$ the number of $k$-ended connected components in $G$ for each $k\geq0$, we show that almost surely, all trees of $G$ have at most two topological ends, $N_0\in\{0,\infty\}$, $N_1\leq2$, and $N_1=2\implies N_2<\infty$. We then construct explicit examples realizing all possibilities compatible with these constraints, yielding a complete classification of the admissible topological structures for $G$. As a second result, we prove that under the additional assumptions that $G$ is non-empty, oriented, out-degree one, with all its directed paths going to infinity along a fixed deterministic direction, the situation reduces to a dichotomy: $G$ consists almost surely of either a unique one-ended tree, or infinitely many two-ended trees. The latter extends a theorem of Chaika and Krishnan (2019), who considered a lattice setting. Our proofs combine classical Burton--Keane type arguments with substantial new conceptual ideas using planar topology, resulting in a robust, unified approach. 
\end{abstract}

\noindent\textbf{MSC2020:} 60D05, 60K35.\\
\textbf{Keywords:} stochastic geometry, random forests, percolation, planar topology.

\section{Introduction}

\label{section_introduction}

Consider a \emph{random geometric forest} $G$ in the plane. Formally,
    \[G=(V,E):(\Omega,\mathcal F,\P)\to\mathcal G\]
is a random graph valued in the configuration space $\mathcal G$ of pairs $(V, E)$, where $V\subset\R^2$ is locally finite and $E\subset V\times V\setminus\{(v, v):v\in V\}$, endowed with the $\sigma$-field generated by the counting functions $(V, E)\mapsto\#(V\cap D)$ and $(V, E)\mapsto\#(E\cap H)$ for bounded Borel sets $D\subset\R^2$ and $H\subset\R^2\times\R^2$. We assume that $G$ is a forest in the sense that its underlying undirected graph is almost surely acyclic, and require that with probability one all vertices of $G$ have finite degree. Note that we allow directed $2$-cycles in $G$ to provide a unified treatment of the case where the graph is effectively undirected, with $(v, v')\in E\iff(v', v)\in E$. By contrast, we say that $G$ is \emph{oriented} if for all $(v, v')\in E$, $(v',v)\notin E$.

Throughout this paper, we focus on the undirected structure of $G$, while certain assumptions will require us to consider directed edges. We write $v\leftrightarrow v'$ for vertices $v,v'\in V$ if there exists an undirected path between them in $G$. The \emph{connected components} of $G$ are then defined as equivalence classes of vertices under this relation. For each connected component $C\subset V$, a \emph{topological end} is an equivalence class of infinite undirected simple paths in $C$ under the relation of \emph{coalescence}, that is, two such elements are equivalent if their vertex sets differ by only finitely many vertices. We then say that $C$ is $n$-ended, with $n\in\Z_+\cup\{\infty\}$, if the number of topological ends in $C$ is $n$.

A fundamental problem is to understand the topological structure of random graphs by determining, under general assumptions, the number of $n$-ended connected components they possess for each $n\in\Z_+\cup\{\infty\}$. Classical tools to address such questions include Burton–Keane type arguments~\cite{burton_and_keane}, various forms of mass-transport principles, and, in some cases, planarity. As an example, one may cite the influential work concerning group-invariant percolation on graphs by Benjamini, Lyons, Peres, and Schramm~\cite{BLPS}, which establishes several results under mild hypotheses. In this spirit, we study the problem in the setting of random forests. Since, without further restrictions, essentially any behavior is possible, we impose the following additional natural hypotheses throughout this work.

\begin{assumption}
    \label{assumptions}
    The random geometric forest $G$ satisfies the following conditions:
    \begin{itemize}
        \item\textbf{Planarity.} For any pair of edges $(u,u')$ and $(v, v')$ in $E$, the corresponding line segments in $\R^2$ can only intersect at a common extremity when they are distinct.
        \item\textbf{Stationarity and ergodicity.} The distribution of $G$ is stationary and ergodic with respect to translations of $\R^2$, i.e., there exists a measurable, measure-preserving, ergodic action $(\theta_s)_{s\in\R^2}$ on $\Omega$ such that $G\circ\theta_s=s+G$ for all $s\in\R^2$, where $s+G$ denotes a geometric graph translation.
        \item\textbf{Finite edge intensity.} The edge set of the graph admits a finite intensity in the sense that 
            \begin{equation}
                \label{eq_Lambda_def}
                \Lambda\coloneqq\E[\#\{(u, v)\in E:[u,v]\cap[0,1]^2\neq\emptyset\}]<\infty.
            \end{equation}
    \end{itemize}
\end{assumption}

\begin{remark}
    \label{remark_ergodicity_stationarity}
    The \emph{ergodicity} requirement may be relaxed by the ergodic decomposition theorem for stationary measures, up to treating the corresponding factors separately. Moreover, \emph{stationarity and ergodicity} may be taken with respect to $\Z^2$-translations: by enlarging the probability space and applying an independent uniform random shift in $[0,1]^2$ to the graph, one recovers the condition as originally formulated.
\end{remark}

These conditions, up to Remark~\ref{remark_ergodicity_stationarity}, hold for all standard examples, as will be illustrated later.

\subsection*{Main results}

Our first result provides a classification of the possible topological structures for $G$.

\begin{theoremI}
    \label{thm_main_undirected}
    With probability one, every connected component of $G$ is a tree with at most $2$ topological ends. Let $N_0, N_1$ and $N_2$ respectively denote the number of finite, one-ended and two-ended connected components of $G$, taking values in $\Z_+\cup\{\infty\}$. Then, $(N_0,N_1,N_2)$ is constant on a full-probability event, and the following holds almost surely:
        \begin{itemize}
            \item $N_0\in\{0,\infty\}$,
            \item $N_1\leq 2$, and
            \item $N_1=2\implies N_2<\infty$.
        \end{itemize}
    Moreover, this classification is optimal: every value for $(N_0,N_1,N_2)$ satisfying the above constraints is realized by some model on which the result applies.
\end{theoremI}

Note that the constancy of $(N_0,N_1,N_2)$ on a full-probability event is an immediate consequence of \emph{ergodicity} since this variable is invariant under translations. Beyond this, the main content of Theorem~\ref{thm_main_undirected} is the almost sure statements $N_1\leq2$ and $N_1=2\implies N_2<\infty$. The other assertions are not specific to the planar setting and extend directly to higher dimensions, as will be explained throughout the paper. More precisely, $N_0\in\{0,\infty\}$ a.s.\ is a straightforward consequence of \emph{stationarity} while the fact that with probability one all connected components have at most $2$ ends follows from a classical Burton–Keane type argument involving \emph{trifurcations}, relying on the combinatorial structure of forests together with \emph{stationarity} and \emph{finite edge intensity}. 

In contrast, the bound $N_1\leq2$ a.s.\ is achieved through a delicate Burton–Keane type argument, exploiting the planar structure and relying on the notion of \emph{finite pendant trees} (Definition~\ref{def_backward_tree}, Section~\ref{section_trifurcation}). 

The implication $N_1=2\implies N_2<\infty$ a.s.\ constitutes the most challenging part of the theorem. The proof we provide for this final piece, which enables the optimal classification once examples realizing all the cases are constructed, is, to our knowledge, entirely novel and represents a key conceptual contribution of our work. Roughly, we formalize the intuition that, thanks to \emph{planarity}, when $N_1=2$ the two one-ended components are intertwined so that the space between them looks like a bi-infinite \emph{corridor}, by introducing a notion of \emph{doors}. Using the planar structure, we show that the set of doors can be endowed with a suitable \emph{betweenness} relation induced by a total order isomorphic to $(\Z,<)$. Notably, this requires the development of a general deterministic topological result of independent interest (Theorem~\ref{thm_planar_ordering}, Section~\ref{section_doors}). Then we prove that \emph{stationarity} together with \emph{finite edge intensity} forces two‑ended components to cross every door while each of them can be crossed only finitely many times, yielding the desired conclusion.

As a second result, we establish a dichotomy under additional assumptions, which will be derived from Theorem~\ref{thm_main_undirected} via an elementary yet efficient argument.

\begin{theoremI}
    \label{thm_main_directed}
    Assume additionally that with probability one, $G$ is oriented with out-degree $1$, non-empty and all directed infinite path $(\pi_n)_{n\geq0}$ in $G$ satisfies $\lim_{n\to\infty}\pi_n\cdot e_2=+\infty$, where $e_2\coloneqq (0,1)\in\R^2$ and $\cdot$ denotes the standard scalar product. Then, exactly one of the following holds almost surely:
        \begin{itemize}
            \item $G$ is a single one-ended tree, or
            \item $G$ consists of infinitely many disjoint two-ended trees.
        \end{itemize}
\end{theoremI}

\begin{remark}
    \label{remark_extension_edge_curves}
    The proofs of both Theorems~\ref{thm_main_undirected} and \ref{thm_main_directed} are robust and extend directly to the setting where the line segments associated with the edges of $G$, appearing explicitly in the \emph{planarity} and \emph{finite edge intensity} conditions, are replaced by arbitrary, and possibly random, curves homeomorphic to $[0,1]$.
\end{remark}

It is important to note that the analogue of Theorem~\ref{thm_main_directed} for subgraphs of the $\Z^2$ lattice was previously established by Chaika and Krishnan~\cite{lattice_setting}. Their approach is, however, rather technical and exploits specific features of such setting, which limits its generalization. In fact, as they observe, their argument does not even directly extend to all rotations of the lattice, although they suggest this should be achievable with additional work. In contrast, the present work adopts a different, more conceptual perspective, allowing the result to be established seamlessly in considerably greater generality.

Similarly, some studies on specific models such as~\cite{competition_interface, coalescence_dsf}, following the seminal work of Licea and Newman~\cite{licea_and_newman}, all of which fall within the scope of Theorem~\ref{thm_main_directed} up to Remark~\ref{remark_ergodicity_stationarity}, proved that the graphs they consider are almost surely connected via a Burton–Keane type argument using \emph{planarity} and \emph{local modifications}. All these approaches are nonetheless intrinsically tied to the specificities of each model, and it remains an open problem to determine the minimal general assumptions required to distinguish between the alternatives of Theorems~\ref{thm_main_undirected} and~\ref{thm_main_directed}.

\subsection*{Higher dimensions}

While it is natural to ask what the classification of possible topological structures for $G$ becomes under the analogue of our assumptions beyond dimension $2$, the problem is, perhaps surprisingly, simpler. Indeed, in dimension $d\geq 3$, the \emph{planarity} assumption loses both meaning and rigidity, becoming a \emph{non-crossing property} that is essentially non-restrictive via a superposition principle. More precisely, if $G_1$ and $G_2$ are two independent random geometric graphs satisfying the analogues of our assumptions in dimension $d\geq3$, then their union also satisfies them. To see this, it suffices to verify that with probability $1$, no two line segments of $G_1$ and $G_2$ intersect. Since two line segments can only intersect if all four endpoints lie in a common affine plane, which has zero Lebesgue measure in dimension $d\geq3$, one can then check that this holds thanks to independence, \emph{stationarity}, and \emph{finite edge intensity} of $G_1$ and $G_2$. Furthermore, examples of a single one-ended tree and of a single two-ended tree satisfying the assumptions up to Remark~\ref{remark_ergodicity_stationarity} are known to exist in any dimension $d\geq3$, see~\cite{existence_single_one_ended, existence_single_two_ended} using that the $\Z^d$ lattice is amenable, unimodular, one-ended, and a Cayley graph. In light of previous remarks, combining these constructions with the superposition principle shows that for $d\geq3$, the classification becomes essentially trivial: each component has at most two ends a.s., the triple $(N_0,N_1,N_2)$ is constant on a full-probability event, $N_0\in\{0,\infty\}$ a.s., and examples exist realizing all possible values for $(N_0, N_1, N_2)$ respecting these constraints. In this sense, the additional exclusions of Theorem~\ref{thm_main_undirected} reflect the genuine rigidity of the planar setting.

Concerning the analogous question for Theorem~\ref{thm_main_directed}, what remains in dimension $d\geq3$ is that $N_2\in\{0,\infty\}$ almost surely, as will be explained later. Furthermore, examples of a single one-ended tree satisfying the analogues of the assumptions in Theorem~\ref{thm_main_directed} up to Remark~\ref{remark_ergodicity_stationarity} are known to exist in any dimension $d\geq3$, see the construction provided in~\cite{nearest_neighbor_graphs}. Therefore, by the same superposition principle, in dimension $d\geq3$, Theorem~\ref{thm_main_directed} takes the following form: almost surely, each component has at most $2$ ends, $N_0=0$, $N_2\in\{0,\infty\}$, and examples exist realizing all possible values for $(N_1, N_2)$ respecting these constraints.

Overall, our results complete the classification of possible topological structures for stationary random forests in all dimensions under mild assumptions.

\subsection*{An important application}

Among the examples covered by our setting is a particular class of oriented out-degree one forests constructed via deterministic translation-covariant rules on a point process, sometimes called \emph{drainage networks}. As expanded below, in this setting our results yield a new route for establishing coalescence, i.e., almost sure connectedness of the corresponding graphs. Formally, let $\N$ be a stationary and ergodic point process in $\R^2$ and define
    \[E_\Psi\coloneqq\{(v, v+\Psi[\mathcal{N}-v]):v\in\N\},\]
where $\Psi$ is a fixed deterministic measurable mapping such that the graph $G_\Psi\coloneq(\N,E_\Psi)$ is well defined and almost surely acyclic. Then, considering $G\coloneqq G_\Psi$, the assumptions of our setup are satisfied whenever the \emph{planarity} condition holds, $\N$ has a \emph{finite intensity} $\lambda\coloneqq\mathbb{E}[\#(\mathcal{N} \cap [0,1]^2)]$, and $G_\Psi$ admits a \emph{finite typical squared out-going edge length}, that is $\E^0[\|\Psi(\N)\|^2]<\infty$, where $\|\cdot\|$ is the standard Euclidean norm and $\E^0$ denotes the expectation with respect to the Palm measure of $\N$. Indeed, one can check that any line segment $[v, v+\Psi(\N-v)]$ with $v\in\N$ that intersects $[0,1]^2$ must satisfy $\|\Psi(\N-v)\|\geq\|v\|-\rho$ where $\rho\coloneqq\sup_{z\in[0,1]^2}\|z\|$, so that recalling the definition of $\Lambda$ in \eqref{eq_Lambda_def},
    \begin{equation*}
        \begin{split}
            \Lambda\leq\E\left[\sum_{v\in\N}\1_{\|\Psi(\N-v)\|\geq\|v\|-\rho}\right]
            &=\lambda\E^0\left[\int_{\R^2}\1_{\|\Psi(\N)\|\geq\|z\|-\rho}\mathrm dz\right]\\
            &=\lambda\E^0\left[\int_0^\infty 2\pi r\1_{\rho+\|\Psi(\N)\|\geq r}\mathrm dr\right]\\
            &=\lambda\pi\E^0\left[\int_0^{\rho+\|\Psi(\N)\|}2r\mathrm dr\right]
            =\lambda\pi\E^0[(\rho+\|\Psi(\N)\|)^2],
        \end{split}
    \end{equation*}
where the first equality follows from the Campbell-Mecke formula. Analogous considerations apply when the construction is based on a marked point process. 

Many works, such as \cite{coalescence_dicrete_drainage_network, discrete_dsf, poisson_trees, lattice_drainage, dsf_coal_vs_dim}, have investigated such drainage networks generalized in arbitrary dimension $d\geq2$. In dimension $d=2$, these models satisfy the assumptions of Theorem~\ref{thm_main_directed} up to Remark~\ref{remark_ergodicity_stationarity}. The main objective of these studies is to establish coalescence almost surely if and only if $d\leq3$, which involves a delicate recurrence analysis of suitable random walks. Several of these papers then show separately that, regardless of dimension, there are almost surely no bi-infinite paths, i.e., every connected component has a single end. While this two-step strategy is natural and effective in arbitrary dimensions, Theorem~\ref{thm_main_directed} shows that in the planar setting it is sufficient to focus solely on either coalescence or absence of bi-infinite paths. In practice, although a unified proof is difficult due to dependencies on model-specific features, ruling out bi-infinite paths is often straightforward regardless of the dimension: in many instances, one can argue by contradiction that if bi-infinite paths existed, a local modification would allow two of them to merge, producing a three-ended component with positive probability, which is impossible by our results. In fact, the alternative behavior described in Theorem~\ref{thm_main_directed}, in which bi-infinite paths exist, can appear pathological, and it is natural to ask whether one can construct an example in which it arises spontaneously, without being explicitly designed.

\subsection*{Other related works}

Beyond drainage networks, our framework encompasses a broad class of constructions, including already well-studied examples ranging from \emph{uniform and minimal spanning trees}~\cite{uniform_spanning_forest, minimal_spanning_forest} to directed geodesic forests arising in first- and last-passage percolation models~\cite{competition_interface, licea_and_newman, euclidean_fpp}\nocite{coalescence_dsf}, both on suitable $\Z^2$ lattice setups and in generalizations or variants based on stationary point processes in $\R^2$. We also mention that the study of random forests has generated a substantial literature, especially in the so-called unimodular setting: see~\cite{unimodular_classification, unimodular_random_trees} for some examples related to the present work.

\subsection*{Structure of the paper}

In Section~\ref{section_trifurcation}, we implement a classical Burton–Keane argument based on trifurcations to show that, almost surely, the connected components of $G$ have at most two ends, and then establish $N_1\leq 2$ almost surely with more delicate considerations involving \emph{planarity} and \emph{finite pendant trees} that we define. In Section~\ref{section_protected_paths}, we derive Theorem~\ref{thm_main_directed} via an efficient elementary proof. In Section~\ref{section_examples}, we construct examples demonstrating the optimality statement of Theorem~\ref{thm_main_undirected}. Finally, in Section~\ref{section_doors}, we prove the key implication $N_1=2\implies N_2<\infty$ almost surely. The argument relies on a general deterministic topological result (Theorem~\ref{thm_planar_ordering}) whose proof is deferred to the concluding Section~\ref{section_planar_ordering}.

\section{Trifurcations}

\label{section_trifurcation}

The goal of this section is to show that with probability one, all components of $G$ have at most two ends, and $N_1\leq 2$. Throughout this work, $\Omega_0$ denotes a full-probability event that may change from line to line on which all almost sure properties encountered hold, starting with the assumptions on $G$. The two results of this section are obtained through different applications of the classical Burton–Keane argument with \emph{trifurcations} formalized in the following lemma and subsequent proposition. Recall that $(\theta_s)_{s\in\R^2}$ denote the shift operators on $\Omega$ mentioned in Assumption~\ref{assumptions}.

\begin{lemma}
    \label{lemma_box_crossings}
    For any integer $n\geq1$, define the random variable that counts the number of edges crossing $\partial([0,n]^2)$ by
        \[\chi_n\coloneqq\#\{(u,v)\in E:[u,v]\cap\partial([0,n]^2)\neq\emptyset\}.\]
    Then $\E[\chi_n]\leq 4\Lambda n$.
\end{lemma}

\begin{proof}
    Fix $n\geq1$. Observe that one can cover $\partial([0,n]^2)$ with $4n$ translations of $[0,1]^2$ by writing $\partial([0,n]^2)\subset\bigcup_{i=0}^{n-1}\{(i, 0),(i, n), (0,i), (n, i)\}+[0,1]^2$, where $+$ denotes the Minkowski sum of sets. Then, denoting $X\coloneq\#\{(u,v)\in E:[u,v]\cap [0,1]^2\neq\emptyset\}$, it comes
        \[\chi_n\leq\sum_{i=0}^{n-1}[X\circ\theta_{-(i,0)}+X\circ\theta_{-(i,n)}+X\circ\theta_{-(0,i)}+X\circ\theta_{-(n,i)}].\]
    Finally, taking the expectation using \emph{stationarity} and the \emph{finite edge intensity} assumptions, one obtains $\E[\chi_n]\leq4n\E[X]=4\Lambda n$, as desired.
\end{proof}

\begin{proposition}[Burton--Keane argument]
    \label{prop_burton_and_keane}
    Fix an event $T\in\mathcal F$, a collection of centers $(c_{i, j})_{i, j\geq1}\in\Z^{\Z_{\geq 0}\times\Z_{\geq 0}}$ and consider for each $n\geq 1$ the random variable $\eta_n$ defined for all $\omega\in\Omega$ by
        \[\eta_n(\omega)\coloneqq\sum_{i=0}^{n-1}\sum_{j=0}^{n-1}\1_{\theta_{-c_{i,j}}(\omega)\in T}.\]
    Assume that there exists an integer $r\geq1$ such that $\chi_{rn}\geq\eta_n$ a.s.\ for all $n\geq 1$. Then, $\P(T)=0$.
\end{proposition}

\begin{proof}
    Fix $n\geq1$. By definition of $\eta_n$ and \emph{stationarity}, $\E[\eta_n]=n^2\P(T)$. Additionally, since $\chi_{rn}\geq\eta_n$ almost surely, one has in particular $\E[\chi_{rn}]\geq\E[\eta_n]$. Together, this implies
        \[\P(T)=\frac{\E[\eta_n]}{n^2}\leq\frac{\E[\chi_{rn}]}{n^2}\leq\frac{4\Lambda r}{n},\]
    where the last inequality follows from Lemma~\ref{lemma_box_crossings}. Letting $n\to\infty$ yields $\P(T)=0$, as desired.
\end{proof}

Informally, in Proposition~\ref{prop_burton_and_keane}, $T$ is typically called a \emph{trifurcation} event, and the random variables $(\eta_n)_{n\geq1}$ count suitably translated occurrences of $T$ that induce a.s.\ at least as many edges crossing the boundaries of boxes of side lengths $(rn)_{n\geq1}$. In what follows we will consider two different possibilities for such an event, namely those defined in \eqref{eq_trifurcation_1} and \eqref{eq_trifurcation_2}.

Now, for each $v\in V$, denote by $\deg_\infty(v)$ the number of vertices $u\in V$ that are neighbors of $v$, i.e.\ $(v,u)\in E$ or $(u, v)\in E$, and such that there exists an infinite simple undirected path starting at $v$ and passing through $u$. The following result is a standard application of the Burton--Keane argument.

\begin{proposition}
    \label{prop_at_most_2_ends}
    Almost surely, for all $v\in V$, $\deg_\infty(v)\leq2$. In particular, with probability one, the number of ends of any connected component of $G$ is at most $2$.
\end{proposition}

\begin{proof}
    Fix $\omega\in\Omega_0$. Consider the sub-graph $G_\infty$ of $G$ induced by the set of vertices 
        \[V_\infty\coloneqq\{v\in V:\deg_\infty(v)\geq2\},\]
    and denote its edge set by $E_\infty$. By construction, no vertex of this graph has degree less than $2$, so it forms a forest without any leaves. Now fix $n\geq 1$ and denote by $G_\infty^n$ the sub-graph of $G_\infty$ induced by the set of vertices
        \[V_\infty^n\coloneqq\bigcup_{(u, v)\in E_\infty^n}\{u,v\}\quad\text{where}\quad E_\infty^n\coloneqq\{(u,v)\in E_\infty:[u,v]\cap[0,n]^2\neq\emptyset\}.\]
    By construction, $G_\infty^n$ has no isolated vertices, and it forms a finite forest since $\#E_\infty^n<\infty$ from \emph{finite edge intensity} combined with \emph{stationarity}. Next, observe that if $v\in V\cap[0,n]^2$, then all its adjacent edges must cross $[0,n]^2$, so if $\deg_\infty(v)\geq2$, then $v$ is an internal vertex of $G_\infty^n$ with degree $\deg_\infty(v)$. Therefore, if $v\in V$ is a leaf of $G_\infty^n$, then $v\notin[0,n]^2$ and, since $v$ is not isolated, there exists a unique $u\in[0,n]^2$ such that $(u,v)$ or $(v,u)$ is an edge of $E_\infty^n$, which must cross $\partial([0,n]^2)$. This shows that the number of edges of $G_\infty^n$ crossing $\partial([0,n]^2)$ is larger than its number of leaves. As it is a classical consequence of the degree-sum formula that, in a finite forest, the leaf count is larger than the number of vertices with degree at least 3, it follows
        \[\chi_n\geq\#\{v\in[0,n]^2:\deg_\infty(v)\geq 3\}.\]
    Now, consider the random variables $(\eta_n)_{n\geq 1}$ associated to the event 
        \begin{equation}
            \label{eq_trifurcation_1}
            T\coloneqq\{\exists u\in[0,1)^2:\deg_\infty(u)\geq3\}
        \end{equation}
    and the centers $(c_{i, j})_{i,j\geq1}\coloneqq[(i, j)]_{i, j\geq1}$ by Proposition~\ref{prop_burton_and_keane}. By construction, for all $n\geq 1$, $\eta_n\leq\#\{v\in[0,n]^2:\deg_\infty(v)\geq 3\}$ and thus $\eta_n\leq\chi_n$, which implies $\P(T)=0$ from Proposition~\ref{prop_burton_and_keane}. It follows by \emph{stationarity} that with probability one, every vertex $v\in V$ satisfies $\deg_\infty(v)\leq 2$, and since $G$ is a.s.\ a forest, none of its connected components can have more than two ends.
\end{proof}

\begin{remark}
    The proof of Proposition~\ref{prop_at_most_2_ends} does not rely on \emph{planarity} and generalizes directly to all dimensions $d\geq2$, for stationary forests in $\R^d$ with finite edge intensity.
\end{remark}

We now turn to the more subtle problem of proving $N_1\leq 2$ almost surely. To that end, we need to introduce some definitions. For any compact $K\subset\R^2$, let $N_1(K)$ denote the number of one-ended connected components that have a vertex in $K$. Define $V_1$ to be the set of vertices that belong to a one-ended component of $G$. Intuitively, the event $\{N_1(K)\geq3\}$, where $K$ is a large enough box, will be our starting point to construct a suitable trifurcation event. In order to exploit the \emph{planarity} assumption there and throughout, we will consider the continuous paths induced by the graph, together with their associated first‑intersection and last‑exit times, defined formally below, as well as the Jordan curve theorem, recalled thereafter with the relevant notation.

\begin{definition}[Continuous infinite path]
    \label{def_continuous_infinite_path}
    For any vertex $v\in V$ such that either $v\in V_1$ or $G$ is oriented with out-degree $1$, let $\pi_v:\R_+\to\R^2$ denote the continuous function obtained by linear interpolation from the unique infinite simple or directed path $\Z_+\to V$ starting at $v$, respectively. Then, for any closed set $F\subset\R^2$, define
        \[t_u^-(F)\coloneqq\inf\{t\geq0:\pi_u(t)\in F\}\quad\text{and}\quad t_u^+(F)\coloneqq\sup\{t\geq0:\pi_u(t)\in F\}\]
    with the convention $\sup\emptyset=-\infty$ and $\inf\emptyset=+\infty$
\end{definition}

\begin{definition}[Jordan curve]
    A subset $J\subset\R^2$ is called a Jordan curve if there exists a continuous bijection $\mathbb S^1\to J$. By the Jordan curve theorem (see \cite{topology_book}), $\R^2\setminus J$ then consists of exactly two open connected components, one bounded and one unbounded, denoted $\Int J$ and $\Ext J$, respectively, both having boundary $J$. The closures of $\Int J$ and $\Ext J$ are denoted by $\clInt J$ and $\clExt J$, respectively.
\end{definition}

We next introduce the notion of \emph{finite pendant trees}, whose key properties are summarized in the subsequent lemma. These objects will be used to define a suitable trifurcation event here, and will prove useful throughout.

\begin{definition}[Finite pendant trees]
    \label{def_backward_tree}
    For any $v\in V$, define the \emph{finite pendant tree} of $v$ as the sub-tree of $G$ induced by the set $\backward(v)$ of vertices $s\in V$ such that $s\leftrightarrow v$ and every infinite undirected simple path starting at $s$ passes through $v$. Additionally, for any compact $K\subset\R^2$, set
        \[\backward(K)\coloneqq\bigcup_{\substack{(u,v)\in E\\ [u,v]\cap K\neq\emptyset}}\backward(v).\]
\end{definition}

\begin{lemma}
    \label{lemma_backward_tree}
    Almost surely, $\backward(v)$ is finite for each $v\in V$. In particular, for every compact set $K\subset \R^2$, $\backward(K)$ is a.s.\ finite. Furthermore, for any compact set $K\subset\R^2$ and vertex $v\in V_1\setminus\backward(K)$, $\pi_v$ avoids $K$.
\end{lemma}

\begin{proof}
    Fix $\omega\in\Omega_0$. For each $v\in V$, since by definition $\backward(v)$ induces a tree that can not contain any infinite branch, König's lemma ensures it is finite since vertices of $G$ have finite degrees. Then, for each compact $K\subset\R^2$, since $\{(u,v)\in E:[u,v]\cap K\neq\emptyset\}$ is almost surely finite from the \emph{finite edge intensity} condition combined with \emph{stationarity}, then $\backward(K)$ is a.s.\ finite. Now fix a compact $K\subset\R^2$ and a vertex $v\in V_1\setminus\backward(K)$. Since $v$ belongs to a one-ended component, then for all $n\geq0$, $v\in\backward[\pi_v(n)]$. In particular, since $v\notin\backward(K)$, for all $n\geq 0$ one must have $[\pi_v(n),\pi_v(n+1)]\cap K=\emptyset$. Hence, $\pi_v$ avoids $K$, which concludes the proof.
\end{proof}

We can now use all of these definitions to conclude the section.

\begin{proposition}
    \label{prop_N1_leq_2}
    Almost surely, $G$ does not contain more than $2$ one-ended connected components.
\end{proposition}

\begin{figure}[ht]
    \label{fig_planar_one_ended_trifurcations}
    \centering
    \includegraphics[width=0.35\linewidth]{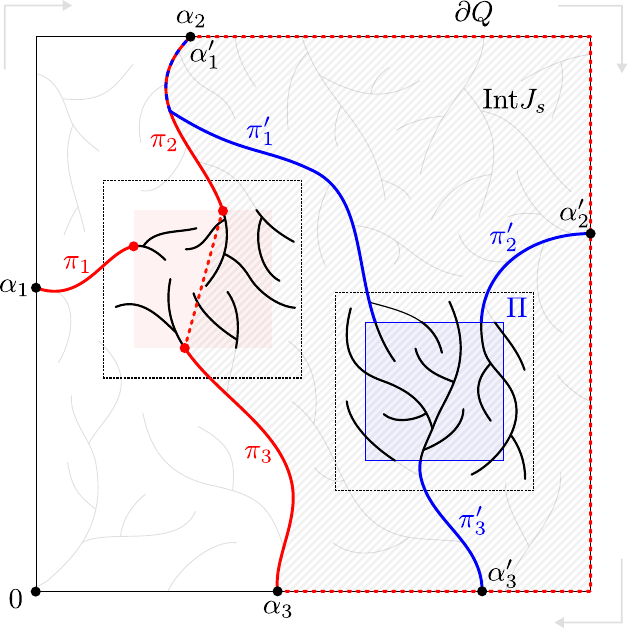}
    \caption{Illustration behind the proof of Proposition~\ref{prop_N1_leq_2}. Dotted squares represent $c_{i, j}+(-k,k+\ell)^2$ and $c_{i', j'}+(-k,k+\ell)^2$ while bold lines represent edges of $\backward(c_{i, j}+[0,k]^2)$ and $\backward(c_{i', j'}+[0,k]^2)$. Here, $s=2$ and $\{\alpha_r'\}_{r\in\{1,2,3\}}\subset [\alpha_2,\alpha_3]$.}
\end{figure}

\begin{proof}
    Let us first reduce the problem to a deterministic question. For all integers $k,\ell\geq 1$, set
        \begin{equation}
            \label{eq_trifurcation_2}
            T_{k,\ell}\coloneqq\{N_1([0,k]^2)\geq3\}\cap\{\backward([0,k]^2)\subset(-\ell, k+\ell)^2\}.
        \end{equation}
    Since $\backward([0,k]^2)$ is almost surely finite for each $k\geq1$, using \emph{stationarity}, one gets
        \[\P(N_1\geq3)=\lim_{k\to\infty}\P[N_1([-\tfrac k2,\tfrac k2]^2)\geq 3]=\lim_{k\to\infty}\P[N_1([0,k]^2)\geq3]=\lim_{k\to\infty}\lim_{\ell\to\infty}\P(T_{k,\ell}).\]
    It is therefore sufficient to check that $\P(T_{k,\ell})=0$ for each $k,\ell\geq 1$. From now on, fix $k,\ell\geq 1$. Consider the random variables $(\eta_n)_{n\geq 1}$ associated via Proposition~\ref{prop_burton_and_keane} to the event $T\coloneq T_{k,\ell}$ and the centers defined for all $i, j\geq1$ by $c_{i, j}\coloneq([k+2\ell]i, [k+2\ell]j)$. Then, Proposition~\ref{prop_burton_and_keane} ensures it is enough to verify that $\chi_{(k+2\ell)n}\geq\eta_n$ a.s.\ for all $n\geq 1$. The problem is now purely geometric. 
    
    Now fix $n\geq 1$ and $\omega\in\Omega_0$. To prove $\chi_{(k+2\ell)n}\geq\eta_n$, let us introduce the following notation. Define $Q\coloneq[0,(k+2\ell)n]^2$ and let $I$ denote the set of index pairs $(i, j)$ between $0$ and $n-1$ such that $\theta_{-c_{i, j}}(\omega)\in T$, so that $\eta_n=\#I$. For each $(i, j)\in I$, one can by definition choose three vertices $[v_s(i, j)]_{s\in\{1,2,3\}}$ in $c_{i, j}+[-\ell, k+\ell]^2$ belonging to different one-ended components of $G$. For each $(i, j)\in I$, set
        \[\pi_{i, j}^{s}:[0,1]\to\pi_{v_s(i, j)}\left(\frac{t-t^0_{i, j}}{t^1_{i, j}-t^0_{i, j}}\right)\quad\text{where}\quad\begin{cases}
         t^0_{i, j}\coloneq t_{v_s(i, j)}^-(\partial Q)\quad\text{and}\\
        t^1_{i, j}\coloneq t_{v_s(i, j)}^+(c_{i, j}+[-\ell,k+\ell]^2).
        \end{cases}\]
    In words, $\pi_{i, j}^s$ denote the re-parametrized portion of the continuous path $\pi_{v_s(i, j)}$ between its last exit of $c_{i, j}+[-\ell,k+\ell]^2$ and first intersection with $\partial Q$. Next, define $f:[0,1)\to\partial Q$ as the continuous bijection that parametrize the boundary of the box $Q$ at constant speed in clockwise order around its center, starting from $0$. For each $(i,j)\in I$ and $s\in\{1,2,3\}$, denote
        \[\alpha_s(i, j)\coloneq f^{-1}\circ\pi_{i, j}^s(1)\in[0,1).\]
    Possibly after permuting the points $[v_s(i, j)]_{s\in\{1,2,3\}}$, let us assume that
        \[\alpha_1(i, j)<\alpha_2(i, j)<\alpha_3(i, j)\quad\text{for all $(i, j)\in I$}.\]
    Now that the relevant objects are defined, observe that since $\{\pi_{i, j}^2(1):(i, j)\in I\}$ correspond to intersection points of paths with $\partial Q$, the number of edges crossing $\partial Q$ must be at least $\{\alpha_2(i, j):(i, j)\in I\}$. Therefore, to show $\chi_{(k+2\ell)n}\geq\eta_n=\#I$, it suffices to establish $\#I=\#\{\alpha_2(i, j):(i, j)\in I\}$, that is, the map
        \[\alpha_2:I\to[0,1),~(i, j)\mapsto\alpha_2(i, j)\]
    is injective. Let us then fix two distinct elements $(i, j)$ and $(i', j')$ in $I$, and denote
        \[\alpha_s\coloneqq \alpha_s(i, j) \quad\text{and}\quad \alpha_s'\coloneqq \alpha_s(i', j') \quad\text{for all $s\in\{1,2,3\}$}\]
    for convenience. It remains to show that $\alpha_2\neq\alpha_2'$. To do so, observe that it is enough to verify that the set $\{\alpha_s'\}_{s\in\{1,2,3\}}$ lies entirely within one of $[\alpha_1,\alpha_2]$, $[\alpha_2,\alpha_3]$ or $[0,\alpha_1]\cup[\alpha_3,1)$. Indeed, containment in $[\alpha_1,\alpha_2]$ or $[\alpha_2,\alpha_3]$ forces either $\alpha_2'<\alpha_3'\leq \alpha_2$ or $\alpha_2'>\alpha_1'\geq \alpha_2$, implying $\alpha_2\neq \alpha_2'$, and if $\{\alpha_s'\}_{s\in\{1,2,3\}}\subset [0,\alpha_1]\cup[\alpha_3,1)$, then either $\alpha_2'\geq \alpha_3>\alpha_2$ or $\alpha_2'\leq \alpha_1<\alpha_2$, yielding the same conclusion. Let us now prove this last property. Set $\pi_s\coloneqq\pi_{i, j}^s$ and $\pi_s'\coloneqq\pi_{i', j'}^s$ for each $s\in\{1,2,3\}$. As illustrated in Figure~\ref{fig_planar_one_ended_trifurcations}, it appears intuitive that $(\pi_s)_{s\in\{1,2,3\}}$ allow to divide the set $Q\setminus(c_{i, j}+[0,k]^2)$ into three regions whose restriction to $\partial Q$ are precisely $[\alpha_1,\alpha_2]$, $[\alpha_2,\alpha_3]$ and $[0,\alpha_1]\cup[\alpha_3,1)$, with one containing all the images of $(\pi_s')_{s\in\{1,2,3\}}$, yielding the desired result.\\
    
    The remaining part of the proof is dedicated to making this heuristic rigorous. Fix $s\in\{1,2\}$ and consider the set
        \[J_s\coloneqq\pi_s([0,1])\cup[\pi_s(0),\pi_{s+1}(0)]\cup\pi_{s+1}([0,1])\cup f([\alpha_s,\alpha_{s+1}]).\]
    By construction, $J_s$ is a Jordan curve. Let $r\in\{1,2,3\}$. Since both $\pi_s$ and $\pi_{s+1}$ corresponds to portions of continuous paths in one-ended components of $G$ that start outside $c_{i', j'}+(-\ell,k+\ell)^2$, then as $\backward(c_{i', j'}+[0,k]^2)\subset c_{i', j'}+(-\ell,k+\ell)^2$ by definition of $(i, j)\in I$, it follows that $\pi_s$, $\pi_{s+1}$, and hence $J_s$, avoid $c_{i, j}+[0,k]^2$ from Lemma~\ref{lemma_backward_tree}. Moreover, by \emph{planarity}, if there exists $t\in[0,1]$ such that $\pi'_r(t)\in\pi_s([0,1])\cup\pi_{s+1}([0,1])$, then $\pi_r'([t,1])\subset \pi_s([0,1])\cup\pi_{s+1}([0,1])$, which implies that $\pi_r'([0,1])\setminus J_s$ is connected. Together, setting
        \[\Pi\coloneqq(c_{i', j'}+[0,k]^2)\cup\bigcup_{r\in\{1,2,3\}}\pi_r'([0,1]),\]
    this proves that $\Pi\setminus J_s$ is connected. Therefore, one has either $\Pi\subset\clInt J_s$ or $\Pi\subset\clExt J_s$ and since $\pi_r'[(0,1)]\subset Q\setminus\partial Q$ by construction for each $r\in\{1,2,3\}$, then setting
        \[S_\Int^s\coloneqq(\partial Q)\cap\partial(\clInt J_s\cap Q\setminus\partial Q)\quad\text{and}\quad S_\Ext^s\coloneqq(\partial Q)\cap\partial(\clExt J_s\cap Q\setminus\partial Q),\]
    it comes that the exit points must satisfy either $[\pi'_r(1)]_{r\in\{1,2,3\}}\subset S_\Int^s$ or $ [\pi'_r(1)]_{r\in\{1,2,3\}}\subset S_\Ext^s$. Applying $f^{-1}$, this implies one has either
        \begin{equation}
            \label{eq_exit_points_int_or_ext}
            [\alpha'_r(1)]_{r\in\{1,2,3\}}\subset f^{-1}(S_\Int^s)\quad\text{or}\quad [\alpha'_r(1)]_{r\in\{1,2,3\}}\subset f^{-1}(S_\Ext^s).
        \end{equation}
    It remains to identify the sets $f^{-1}(S_\Int^s)$ and $f^{-1}(S_\Ext^s)$. More precisely, it suffices to verify that $f^{-1}(S_\Int^s)\subset[\alpha_s,\alpha_{s+1}]$ and $f^{-1}(S_\Ext^s)\subset[0,1)\setminus(\alpha_s,\alpha_{s+1})$. Indeed, if this holds, then together with \eqref{eq_exit_points_int_or_ext} one obtains that either $[\alpha'_r(1)]_{r\in\{1,2,3\}}\subset [\alpha_s,\alpha_{s+1}]$ or $[\alpha'_r(1)]_{r\in\{1,2,3\}}\subset[0,1)\setminus(\alpha_s,\alpha_{s+1})$. Finally, taking this assertion with $s=1$ and $s=2$ yields that the set $\{\alpha_r'\}_{r\in\{1,2,3\}}$ lies entirely within one of $[\alpha_1,\alpha_2]$, $[\alpha_2,\alpha_3]$ or $[0,\alpha_1]\cup[\alpha_3,1)$, as desired.
    
    First, observe that 
        \[[0,1)\setminus[\alpha_s,\alpha_{s+1}]\subset f^{-1}([\partial Q]\cap\Ext J_s).\]
    Indeed, if $\alpha\in[0,1)\setminus[\alpha_s,\alpha_{s+1}]$, then $f(\alpha)\in(\partial Q)\setminus J_s$, and since $\R^2\setminus Q$ is connected, unbounded, and avoids $J_s$, then $\R^2\setminus Q\subset\Ext J_s$ and thus $f(\alpha)\in \Ext J_s$, as desired. Therefore, observing that $S_\Int^s\subset(\partial Q)\cap\clInt J_s=(\partial Q)\setminus\Ext J_s$, it follows
        \begin{equation*}
            \label{eq_S_int}
            f^{-1}(S_\Int^s)\subset f^{-1}([\partial Q]\setminus \Ext J_s)\subset[\alpha_s, \alpha_{s+1}].
        \end{equation*}
    It remains to check that $f^{-1}(S_\Ext^s)\subset[0,1)\setminus(\alpha_s,\alpha_{s+1})$. To that end, let us show that one has $(\alpha_s, \alpha_{s+1})\subset f^{-1}([\partial Q]\setminus S_\Ext^s)$. Fix $\alpha\in (\alpha_s, \alpha_{s+1})$. Then, $z\coloneqq f(\alpha)$ must be at positive distance from the compact set $J_s\setminus f[(\alpha_s, \alpha_{s+1})]$, hence there exists $\varepsilon\in(0,1)$ such that
        \[(z+[-\varepsilon,\varepsilon]^2)\setminus J_s=(z+[-\varepsilon,\varepsilon]^2)\setminus\partial Q.\]
    Next, observe that $(z+[-\varepsilon,\varepsilon]^2)\setminus\partial Q$ consists of exactly two connected component, namely $U_\Int\coloneqq(z+[-\varepsilon,\varepsilon]^2)\cap Q\setminus\partial Q$ and $U_\Ext\coloneqq (z+[-\varepsilon,\varepsilon]^2)\setminus Q$. Therefore, $(z+[-\varepsilon,\varepsilon]^2)\setminus J_s$ must also consist of two connected components, $(z+[-\varepsilon,\varepsilon]^2)\cap\Int J_s$ and $(z+[-\varepsilon,\varepsilon]^2)\cap\Ext J_s$. Since $U_\Ext\subset(z+[-\varepsilon,\varepsilon]^2)\cap\Ext J_s$ from $U_\Ext\subset\R^2\setminus Q\subset\Ext J_s$, it follows that $U_\Ext=(z+[-\varepsilon,\varepsilon]^2)\cap\Ext J_s$ by identifying the connected components of $(z+[-\varepsilon,\varepsilon]^2)\setminus J_s$ and $(z+[-\varepsilon,\varepsilon]^2)\setminus\partial Q$. Thus,
        \begin{equation*}
            \begin{split}
                (z+[-\varepsilon,\varepsilon]^2)\cap\clExt J_s\cap Q\setminus\partial Q&=(z+[-\varepsilon,\varepsilon]^2)\cap\clExt J_s\cap Q\setminus\partial J_s\\
                &=(z+[-\varepsilon,\varepsilon]^2)\cap\Ext J_s\cap Q\\
                &=U_\Ext\cap Q\\
                &=\emptyset,
            \end{split}
        \end{equation*}
    which shows $z\notin\partial (\clExt J_s\cap Q\setminus \partial Q)$, and hence $\alpha\in f^{-1}([\partial Q]\setminus S_\Ext^s)$. This proves the inclusion $(\alpha_s, \alpha_{s+1})\subset f^{-1}([\partial Q]\setminus S_\Ext^s)$. As a consequence,
        \begin{equation*}
            \label{eq_S_ext}
            f^{-1}(S_\Ext^s)\subset[0,1)\setminus (\alpha_s, \alpha_{s+1}).
        \end{equation*}
    This concludes the proof.
\end{proof}

\section{Derivation of the dichotomy result}

\label{section_protected_paths}

In this section, we derive Theorem~\ref{thm_main_directed} from the results already obtained in the previous section. To that end, let $V_2$ denote the set of vertices that belong to a two-ended component of $G$. Also set $V^-\coloneq V\cap(\R\times\R_-)$, $V_1^-\coloneq V_1\cap(\R\times\R_-)$ and $V_2^-\coloneq V_2\cap(\R\times\R_-)$. Under the hypothesis of Theorem~\ref{thm_main_directed}, for all $\omega\in\Omega_0$ and $v\in V^-$, one has $\pi_v(0)\cdot e_2\leq 0$ and $\lim_{t\to\infty}\pi_v(t)\cdot e_2=+\infty$, so that $t^+_v(\R\times\R_-)\in\R_+$ and we can define
    \[\pi_v^+\coloneq\pi_v(\cdot - t^+_v[\R\times\R_-])\quad\text{and}\quad x_v\coloneqq \pi_v^+(0)\cdot e_2.\]
In words, $\pi_v^+$ is the continuous path induced by $\pi_u$ after its last exit of $\R\times\R_-$. We will prove Theorem~\ref{thm_main_directed} from the following key lemma and a subsequent elementary result.

\begin{lemma}
    \label{prop_protected_path}
    Under the hypotheses of Theorem~\ref{thm_main_directed}, almost surely, for all $(u, v)\in V_1^-\times V^-$, $u\leftrightarrow v$.
\end{lemma}

\begin{proof}
    First, observe that up to also considering the graph reflected with respect to the first coordinate, defining for each integer $k\geq 1$ the event
        \[A_k\coloneq\{\exists (u,v)\in V_1^-\times V^-,~-k<x_u<x_v<k\quad\text{and}\quad v\nleftrightarrow u\},\]
    it suffices to show $\P(A_k)=0$ for each $k\geq1$. From now on, fix $k\geq1$ and define the random variable
        \[X:\omega\mapsto\sum_{i\in\Z}\1_{\theta_{-2kie_1}(\omega)\in A_k}.\]
    Then, showing $N_1\geq X$ almost surely is enough to conclude. Indeed, if this hold, then by Proposition~\ref{prop_at_most_2_ends} and \emph{stationarity}, $2\geq\E[N_1]\geq\E[X]=\sum_{i\in\Z}\P(A_k)$, which implies $\P(A_k)=0$. From now on, let us fix a realization $\omega\in\Omega_0$. Set $I\coloneqq\{i\in\Z:\theta_{-2kie_1}(\omega)\in A_k\}$. By definition, for each $i\in\Z$, there exists $(u_i,v_i)\in V_0^-\times V^-$ with $k(2i-1)<x_{u_i}<x_{v_i}<k(2i+1)$ such that $v_i\nleftrightarrow u_i$. Let us conclude by showing that for all $i,i'\in I$ with $i<i'$, $u_i$ and $u_{i'}$ belongs to different one-ended component of $G$, so that $N_1\geq\#I=X$. Fix $i,i'\in I$ such that $i<i'$ and assume by contradiction that $u_i\leftrightarrow u_{i'}$. Then, the directed infinite paths started at $u_i$ and $u_{i'}$ coalesce, so that
        \[J\coloneqq[\pi^+_{u_i}(0), \pi^+_{u_{i'}}(0)]\cup\overline{\pi^+_{u_i}(\R_+)\triangle \pi^+_{u_{i'}}(\R_+)}\]
    forms a Jordan curve. Now observe that by construction, since $i'\geq i+1$,
        \[x_{u_i}<x_{v_i}<x_{u_{i'}},\]
    so that $\pi^+_{v_i}(0)\notin\overline{\pi^+_{u_i}(\R_+)\triangle \pi^+_{u_{i'}}(\R_+)}$ and there exists $\varepsilon>0$ such that the ball 
        \[B\coloneqq\{z\in\R^2:\|z-\pi^+_{v_i}(0)\|<\varepsilon\}\]
    satisfies $B\setminus J=B\setminus(\R\times\{0\})$. Additionally, since $\R\times\R_-^*$ is connected, unbounded, and avoids $J$, then $\R\times\R_-^*\subset\Ext J$. Together, identifying the two connected components of $B\setminus(\R\times\{0\})$ to those of $B\setminus(\R\times\{0\})$, this yields
        \[B\cap (\R\times\R_+^*)=B\cap\Int J.\]
    Therefore, since $\pi^+_{v_i}(\R_+^*)\subset \R\times\R_+^*$ and $\pi^+_{v_i}(\R_+^*)$ avoids both $\pi^+_{u_i}(\R_+)$ and $\pi^+_{u_i}(\R_+)$ as $v_i\nleftrightarrow u_i\leftrightarrow u_{i'}$, this yields $\pi^+_{v_i}(\R_+^*)\subset\Int J$, which is absurd since $\pi^+_{v_i}(\R_+^*)$ is unbounded. This completes the proof.
\end{proof}

\begin{lemma}
    \label{lemma_N2_directed_dichotomy}
    Under the hypotheses of Theorem~\ref{thm_main_directed}, $N_2\geq1\implies N_2=\infty$ almost surely.
\end{lemma}

\begin{proof}
    Since $N_2$ is constant on an event of full-probability, let us assume that $N_2\geq 1$ with probability one. Since $N_2\geq 1$ implies $V_2\neq\emptyset$, and $V_2=\bigcup_{n\geq 0}(ne_2+V_2^-\circ\theta_{-ne_2})$, \emph{stationarity} ensures that $\P(V_2^-\neq\emptyset)=1$. Now, define
        \[X\coloneqq\{x_v:v\in V_2^-\}.\]
    By construction, for each $v,v'\in V_2^-$, $v\leftrightarrow v'$ if and only if $v'\in\pi_v(\R_+)$ or $v\in\pi_{v'}(\R_+)$, i.e., $x_v=x_{v'}$, so that
        \[N_2\geq\#X\geq\1_{V_2^-\neq\,\emptyset}\]
    almost surely. Then, denoting $X_i\coloneq X\cap[i,i+1)$ for all $i\in\Z$, it comes
        \[\E[N_2]\geq\E[\#X]=\sum_{i\in\Z}\E[X_i]=\sum_{i\in\Z}\E[X_0]\in\{0,\infty\},\]
    by \emph{stationarity}. Together with $\E[\#X]\geq\P(V^-_2\neq\emptyset)=1$, this implies that $\E[N_2]=\infty$ and thus $N_2=\infty$ since $N_2$ is constant on an event of full-probability.
\end{proof}

\begin{remark}
    The proof of Lemma~\ref{lemma_N2_directed_dichotomy} does not rely on \emph{planarity} and generalizes directly to higher dimensions.
\end{remark}

\begin{proof}[Proof of Theorem~\ref{thm_main_directed}]
    First, observe that with probability one, $G$ contains no finite components since it is an out-degree one forest. If $N_1\geq 1$ almost surely, then, since $V_1=\bigcup_{n\geq 0}(ne_2+V_1^-\circ\theta_{-ne_2})$, \emph{stationarity} gives that with probability one, $V_1^-\neq\emptyset$. Together with Lemma~\ref{prop_protected_path}, one gets that all vertices in $V^-$ belong to the same one-ended of $G$. Using \emph{stationarity} again, since $V=\bigcup_{n\geq 0} (ne_2+V^-\circ\theta_{-ne_2})$, this shows that $N_1\geq 1$ implies $N_1=1$ and $N_2=0$, almost surely. Since, with probability one, all components of $G$ have at most two ends from Proposition~\ref{prop_at_most_2_ends} and $N_2\geq 1\implies N_2=\infty$ almost surely from Lemma~\ref{lemma_N2_directed_dichotomy}, this establishes the result.
\end{proof}

\section{Examples}

\label{section_examples}


In this section, we construct examples showing the optimality of Theorem~\ref{thm_main_undirected}. Our construction is based on the \emph{uniform spanning tree} of the lattice $\Z^2$, denoted $G_\mathrm{UST}=(\Z^2, E_\mathrm{UST})$ on a probability space $(\Omega,\mathcal F,\P)$. We refer to \cite{uniform_spanning_forest} for further details. It is well known that $G_\mathrm{UST}$ is almost surely a one-ended tree and that its law is stationary and ergodic with respect to translations of $\Z^2$. Using Remark~\ref{remark_ergodicity_stationarity}, this already provides an example realizing $(N_0, N_1, N_2)=(0,1,0)$. We will now deterministically construct other examples from this base case to realize the remaining possibilities.

We begin by introducing the \emph{dual graph} of $G_\mathrm{UST}$, denoted $G_\mathrm{UST}$. Its vertices are the points $(\frac12,\frac12)+\Z^2$, and its edges consist of pairs $(u, v)$ with $\|u-v\|=1$ such that $[u, v]$ does not intersect any line segment corresponding to an edge in $G_\mathrm{UST}$. Next, we define \emph{contour} of a sub-graph of the $\Z^2$ lattice.

\begin{definition}[Contour]
    If $\mathcal Z=(V_\mathcal Z,E_\mathcal Z)$ is a sub-graph of the $\Z^2$ lattice and
    $\varepsilon\in(0,\tfrac12)$, its contour at distance $\varepsilon$ is defined as the graph
    $\mathrm{Cont}[\mathcal Z,\varepsilon]$ with vertex set
    $V_\mathcal Z+\varepsilon\Delta$, where $\Delta\coloneqq\{-1,1\}^2$, and whose edges consist of 
        \begin{itemize}
            \item pairs $(v+\varepsilon\delta,v+\varepsilon\delta')$ with $v\in V_\mathcal Z$
            and distinct $\delta,\delta'\in\Delta$ such that
            $\bigl(v,\,v+\tfrac{\delta+\delta'}2\bigr)\notin E_\mathcal Z$,
            and
            \item pairs $(v+\varepsilon\delta,v'+\varepsilon\delta')$ with $(v,v')\in E_\mathcal Z$
            and $\delta,\delta'\in\Delta$ satisfying
            $\tfrac{\delta-\delta'}2=v'-v$.
        \end{itemize}
\end{definition}

We now establish the key property that the contours of $G_\mathrm{UST}$ are almost surely bi-infinite paths.

\begin{lemma}
    \label{lemma_bi_infinite_ray}
    For all $\varepsilon\in(0,\tfrac{1}{2})$, almost surely, the graph 
    $\mathrm{Cont}[G_\mathrm{UST}, \varepsilon]$ is isomorphic to the $\Z$ lattice, hence in particular a two-ended tree.
\end{lemma}

\begin{figure}[ht]
    \label{fig_contour}
    \centering
    \includegraphics[width=0.4\linewidth]{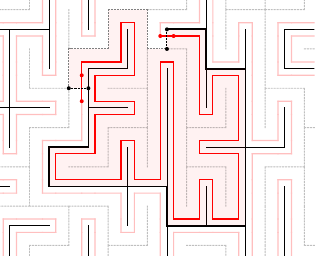}
    \caption{Simulation illustrating the proof of Lemma~\ref{lemma_bi_infinite_ray} realized using the Broder-Aldous algorithm (see \cite{broder_aldous}). The edges of $G_\mathrm{UST}$ and $G_\mathrm{UST}^*$ are represented by solid black and dotted gray lines, respectively, while those of $\mathrm{Cont}[G_\mathrm{UST}, \varepsilon]$ appear in red. Red and black dots indicate the vertices $\{w_1,w_1',w_2,w_2'\}$ and the points $\{z_1,z_1^*,z_2,z_2^*\}$, respectively. The translucent red region represents $\Int J$.}
\end{figure}

\begin{proof}
    An illustration to follow the proof is given in Figure~\ref{fig_contour}. Fix $\varepsilon\in(0,\tfrac12)$ and $\omega\in\Omega_0$. By construction,
    $\mathrm{Cont}[G_\mathrm{UST}, \varepsilon]$ is a $2$-regular infinite graph.
    Thus, it suffices to prove that it is connected. Let $C$ and $C^*$ denote the connected subsets of $\R^2$ obtained by taking the unions of the line segments corresponding to the edges of $G_\mathrm{UST}$ and $G_\mathrm{UST}^*$, respectively. Fix two distinct edges
        \[(w_1,w_1')=(v_1+\varepsilon\delta_1,\,v_1'+\varepsilon\delta_1')\quad\text{and}\quad(w_2,w_2')=(v_2+\varepsilon\delta_2,\,v_2'+\varepsilon\delta_2')\]
    in $\mathrm{Cont}[G_\mathrm{UST}, \varepsilon]$, where
    $v_1,v_1',v_2,v_2'\in\Z^2$ and $\delta_1,\delta_1',\delta_2,\delta_2'\in\Delta$. Let us show that $w_1$ and $w_2$ are connected in $\mathrm{Cont}[G_\mathrm{UST},\varepsilon]$ to conclude the proof. For each $i\in\{1,2\}$, set
        \[z_i\coloneqq \frac{v_i+v_i'}{2}\qquad\text{and}\qquad z_i^*\coloneqq z_i+\frac{\delta_i+\delta_i'}{4}.\]
    Then one checks that $z_i\in C$, $z_i^*\in C^*$, $[z_i,z_i^*]\setminus\{z_i,z_i^*\}$ avoids $C\cup C^*$, and $[w_i,w_i']$ is the only line segment of an edge in $\mathrm{Cont}[G_\mathrm{UST}, \varepsilon]$ intersecting $[z_i,z_i^*]$. Since $C$ and $C^*$ are connected sets corresponding to trees, there exist continuous simple paths $f:[0,1]\to C$ from $z_1$ to $z_2$ and $f^*:[0,1]\to C^*$ from $z_1^*$ to $z_2^*$. The set
        \[J\coloneqq [z_1,z_1^*]\cup f^*([0,1])\cup[z_2^*,z_2]\cup f([0,1])\]
    then forms a Jordan curve. One can check that by construction, the points $w_1$ and $w_1'$ lie in different connected components of $\R^2\setminus J$. Assume without loss of generality that $w_1\in\Ext J$ and $w_1'\in\Int J$. Since $\Int J$ contains only finitely many vertices of the $2$-regular graph $\mathrm{Cont}[G_\mathrm{UST}, \varepsilon]$, the unique path in $\mathrm{Cont}[G_\mathrm{UST}, \varepsilon]$ that starts by entering $\Int J$ along $(w_1,w_1')$ and never immediately retraces an edge must eventually exit $\Int J$ through $(w_2,w_2')$ or $(w_2',w_2)$. This concludes the proof.
\end{proof}

\subsection*{Cases with $N_2<\infty$}

From this point, many examples can be obtained. Indeed, for all distinct $\varepsilon,\varepsilon'\in(0,\frac{1}{2})$, the line segments corresponding to edges of the a.s.\ two-ended tree $\mathrm{Cont}[G_\mathrm{UST}, \varepsilon]$ intersect none of those in $G_\mathrm{UST}$, $G_\mathrm{UST}^*$, or $\mathrm{Cont}[G_\mathrm{UST}, \varepsilon']$. Consequently, if we denote by  $G_\mathrm{iso}\coloneqq[ (\tfrac13,\tfrac13)+\Z^2,\emptyset]$ the graph consisting of isolated vertices, then any $G$ obtained as a finite union of elements chosen among $G_\mathrm{UST}$, $G_\mathrm{UST}^*$, $(\mathrm{Cont}[G_\mathrm{UST}, \tfrac{1}{n}])_{n\geq4}$, and $G_\mathrm{iso}$ satisfies Assumption~\ref{assumptions} up to Remark~\ref{remark_ergodicity_stationarity}. Then, all cases of Theorem~\ref{thm_main_undirected} with $N_2<\infty$ are realized.

\subsection*{Cases with $N_2=\infty$}

It remains to construct examples with $N_2=\infty$. Note that while the preceding construction relied on taking finite unions, which automatically preserve the \emph{finite edge intensity} condition in Assumption~\ref{assumptions}, allowing infinitely many two-ended components when $N_1=1$ is more delicate. To address this difficulty, we introduce the following notion.

\begin{definition}[Peeling]
    For any graph $G_0$, let $\mathrm{Peel}(G_0)$ denote the graph obtained by removing all leaves of $G_0$. For all $n\geq 1$, let $\mathrm{Peel}^n$ denote the $n$-fold composition of this operation.
\end{definition}

Next, observe that since $G_\mathrm{UST}$ is a.s.\ one-ended, then for each integer $n\geq 0$, the graph $\mathrm{Peel}^n(G_\mathrm{UST})$ is a.s.\ infinite, so that for each $\varepsilon\in(0,\frac12)$, its contour $\mathrm{Cont}[\mathrm{Peel}^n(G_\mathrm{UST}), \varepsilon]$ is a.s.\ isomorphic to the $\Z$ lattice, and hence two-ended. To see this last point, one can proceed by induction, with the base case given by Lemma~\ref{lemma_bi_infinite_ray}. To get the induction step at $n\geq0$, one can check that if $(v_i)_{i\in\Z}$ is an isomorphism from the $\Z$ lattice to $\mathrm{Cont}[\mathrm{Peel}^n(G_\mathrm{UST}), \varepsilon]$, then any sequence $(v_i')_{i\in\Z}$ obtained from $(v_i)_{i\in\Z}$ by discarding the points that are vertices of $\mathrm{Cont}[\mathrm{Peel}^n(G_\mathrm{UST}), \varepsilon]$ while preserving the order of the others forms an isomorphism from the $\Z$ lattice to $\mathrm{Cont}[\mathrm{Peel}^{n+1}(G_\mathrm{UST}), \varepsilon]$. Now, for any increasing application $\phi:\Z_+\to\Z_+$, let us define $G_\phi$ as the graph obtained by union of the
    \[(\mathrm{Cont}[\mathrm{Peel}^{\phi(n)}(G_\mathrm{UST}), \tfrac1n])_{n\geq4}.\]
Then, $G_\phi$ consists a.s.\ of infinitely many two-ended components. Moreover, if $\phi$ can be chosen so that $G_\phi$ satisfies a \emph{finite edge intensity} condition, that is the average number $\Lambda_\phi$ of edges in $G_\phi$ crossing $[-\frac12,\frac12]^2$ is finite, then, up to using Remark~\ref{remark_ergodicity_stationarity}, the graphs $G$ obtained by union of elements among $G_\mathrm{UST}^*$, $G_\phi$ and $G_\mathrm{iso}$ realizes all the remaining cases in Theorem~\ref{thm_main_undirected}.

\begin{figure}[ht]
    \label{fig_peeling_example}
    \centering
    \includegraphics[width=0.33\linewidth]{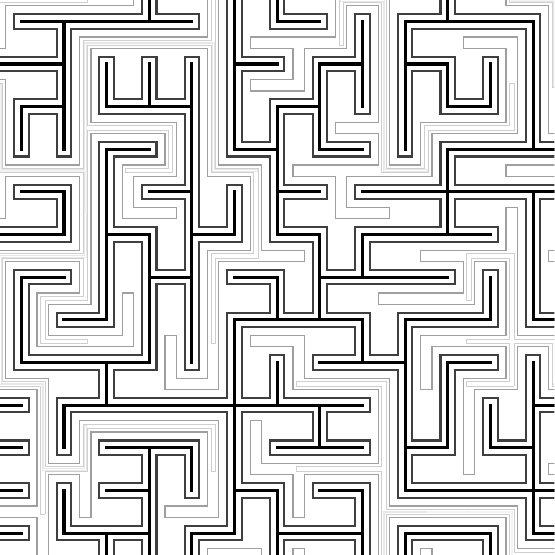}\hspace{.16\linewidth}\includegraphics[width=0.33\linewidth]{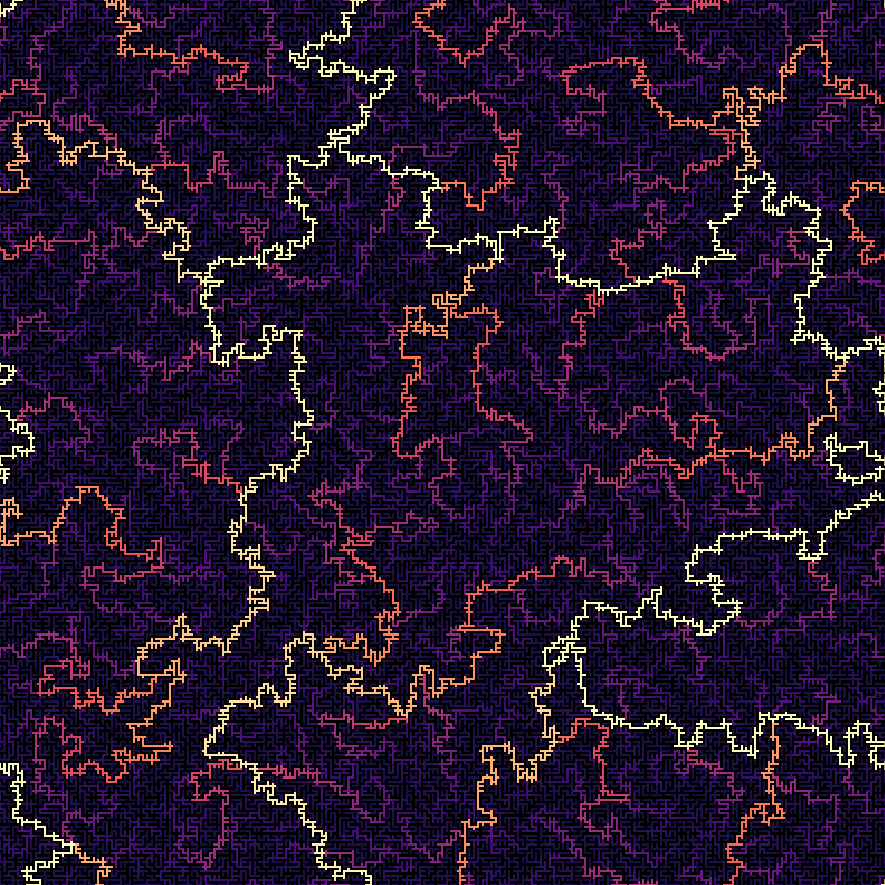}
    \caption{On the left, a simulation of $G_\mathrm{UST}^*$ and $G_\phi$ for some arbitrarily chosen increasing function $\phi:\Z_+\to\Z_+$. On the right, a simulation of $G_\mathrm{UST}$ with each edge $(v,v')$ colored according to $\max\{d_\mathrm{Peel}(v),d_\mathrm{Peel}(v')\}$, where brightness increases with higher values. Qualitatively, this picture corresponds to the edge count per unit square of the graph on the left, viewed at a larger scale.}
\end{figure}

We conclude the section by showing that $\phi$ can indeed be chosen such that $\Lambda_\phi<\infty$. An illustration to follow the argument is provided in Figure~\ref{fig_peeling_example}. For each $n\geq0$, let $V_n$ denote the vertex set of $\mathrm{Peel}^n(G_\mathrm{UST})$, and for all vertex $v\in\Z^2$, define its \emph{peeling depth} in $G_\mathrm{UST}$ by
    \[d_\mathrm{Peel}(v)\coloneqq\inf\{n\geq0:v\notin V_n\}\]
with the convention $\inf\emptyset=\infty$. Observe that for all $v\in\Z^2$, $d_\mathrm{Peel}(v)$ is almost surely finite. To see that, fix $\omega\in\Omega_0$ and $v\in\Z^2$ and consider the depth $n\in\Z_+$ of the finite sub-tree rooted at $v$ induced by $\backward(v)$. Then, one can check that in $\mathrm{Peel}^n(G_\mathrm{UST})$, all points of $\backward(v)$ have been removed except for $v$. Therefore, since $G_\mathrm{UST}$ is one-ended, $d_\infty(v)=1$ so that $v$ is a leaf of $\mathrm{Peel}^n(G_\mathrm{UST})$, and thus $d_\mathrm{Peel}(v)=n+1<\infty$. Now observe that for any $n\geq0$, the number of line segment corresponding to edges of $\mathrm{Cont}[\mathrm{Peel}^n(G_\mathrm{UST}), \varepsilon]$ intersecting $[-\frac12,\frac12]^2$ is null if $0\notin V_n$ and at most $8$ otherwise, so that for any increasing $\phi:\Z_+\to\Z_+$,
    \[\Lambda_\phi\leq8\E\left[\sum_{n\geq0}\1_{0\in V_{\phi(n)}}\right]=8\E\left[\sum_{n\geq0}\1_{d_\mathrm{Peel}(0)>\phi(n)}\right]=8\sum_{n\geq0}\P[d_\mathrm{Peel}(0)>\phi(n)].\]
Finally, since $d_\mathrm{Peel}(0)$ is almost surely finite, then $\lim_{t\to\infty}\P[d_\mathrm{Peel}(0)>t]=0$ and one can choose $\phi$ so that the series in the right hand side of the display above converges, yielding $\Lambda_\phi<\infty$ as desired.

\section{Corridor structure}

\label{section_doors}

The goal of this section is to prove the following proposition, which is the last point of Theorem~\ref{thm_main_undirected}.

\begin{proposition}
    \label{prop_N1_is_2_implies_N2_finite}
    Almost surely, $N_1=2\implies N_2<\infty$.
\end{proposition}

Since $N_1$ is constant on an event of full-probability, let us assume without loss of generality in this section that $N_1(\omega)=2$ for all $\omega\in\Omega_0$. Before turning to the formal argument, we briefly describe the overall strategy. Owing to the planar structure, the region between the two one-ended components may be viewed as a \emph{corridor} equipped with a naturally ordered family of \emph{doors}, isomorphic to $(\Z,<)$. Then, every two-ended component lies in this corridor, and by \emph{finite edge intensity}, each door is crossed by only finitely many of them. This gives two consequences. First, only finitely many two-ended components can cross every door. Second, the existence of components that do not cross every door is almost surely impossible, since one could otherwise identify a smallest minimal or largest maximal crossed door among such components, which is incompatible with stationarity. Together, these observations yield that only finitely many distinct two-ended components can exist when there are two one-ended components. The remaining part of the section is dedicated to making this argument rigorous. Let us start by introducing some definitions. For all $1\leq k<\ell$, define the event
    \[D_{k,\ell}\coloneqq \{N_1([-k,k]^2)=2\}\cap\{\backward([-k,k]^2)\subset(-\ell,\ell)^2\}.\]
Then,
    \[1=\P(N_1=2)=\lim_{k\to\infty}\P[N_1([-k,k]^2)=2]=\lim_{k\to\infty}\lim_{\ell\to\infty}\P(D_{k,\ell}),\]
and we can choose $1\leq k<\ell$ such that
    \[\P(D_{k,\ell})>0.\]
From now on, we fix such values of $k$ and $\ell$. For each $\omega\in D_{k,\ell}$, let $u_1$ denote the westmost-southmost (i.e., minimizing first the horizontal coordinate and then the vertical coordinate) vertex of $V_1\cap [-k,k]^2$. Similarly, let $u_2$ denote the westmost-southmost vertex of $[-k,k]^2$ that belongs to the other one-ended component of $G$. Set
    \[\pi_{u_1}^+\coloneqq\pi_{u_1}[\cdot-t^+_{u_1}([-k, k]^2)]\quad\text{and}\quad\pi_{u_2}^+\coloneqq\pi_{u_2}[\cdot-t^+_{u_2}([-k, k]^2)].\]
In words, $\pi_{u_1}^+$ and $\pi_{u_2}^+$ corresponds to the continuous simple path started at $u_1$ and $u_2$ after their last intersection with $[-k,k]^2$. This is well defined since $\lim_{t\to\infty}\|\pi_{u_1}(t)\|=\infty$ for each $i\in\{1,2\}$. We can now properly define \emph{doors} as follows.

\begin{definition}[Doors]
    Let us define the set of \emph{doors} as the random variable
        \[\D:\omega\in\Omega\mapsto\D(\omega)\coloneqq\{d\in\Z^2:\theta_{-2\ell d}(\omega)\in D_{k,\ell}\}.\]
    For each $d\in\D$ and $i\in\{1,2\}$, set also 
        \[\pi_{d, i}^+\coloneqq 2\ell d+\pi_{u_i}^+(\cdot)\circ\theta_{-2\ell d}.\]
\end{definition}

Intuitively, one can view doors as locations of the space where the two one-ended components of $G$ are close in some suitable sense. Note that although $\D$ may not be almost surely non-empty, one has
    \[\P(\D\neq\emptyset)\geq\P(0\in\D)=\P(D_{k,\ell})>0.\]
We now want to argue that because of the planar structure, the two one-ended components of $G$ induce a natural order on the set of doors. To do so, we first introduce the following definition.

\begin{definition}[Topological lines and portions]
    A subset $\gamma\subset\R^2$ is called a \emph{topological line} if there exists a homeomorphism $\Phi:\R\to\gamma$. When $\gamma\subset\R^2$ is a topological line, a subset $\pi\subset\gamma$ is called a \emph{portion} of $\gamma$ if $\pi$, in the subspace topology of $\gamma$, is open, connected, and relatively compact. Equivalently, $\pi\subset\gamma$ is a portion of the topological line $\gamma\subset\R^2$ if for some and hence any homeomorphism $\Phi:\R\to\gamma$, the set $\Phi^{-1}(\pi)$ is a bounded open interval of $\R$.
\end{definition}

Next, we define topological lines associated with the doors. Fix $\omega\in\Omega_0$. For every $d\in\D$ and $i\in\{1,2\}$, the function $\pi_{d,i}^+$ is continuous, injective, satisfies $\lim_{t\to\infty}\|\pi_{d, i}^+(t)\|=\infty$ and hence realizes an homeomorphism on its image. Therefore, by construction, for each $d\in\D$, the set
    \[\gamma_d\coloneqq\pi_{d,1}^+(\R_+)\cup[\pi_{d,1}^+(0), \pi_{d,2}^+(0)]\cup\pi_{d,2}^+(\R_+)\]
is a topological line. Furthermore, for all distinct $d,d'\in\D$, one has $\gamma_d\neq\gamma_{d'}$ and the difference
    \[\gamma_d\setminus\gamma_{d'}\]
is a portion of $\gamma_d$. Indeed, if $d,d'\in\D$ are distinct, then both $\pi_{d,1}^+$ and $\pi_{d,2}^+$ avoid $2\ell d'+[-k,k]^2$ using Lemma~\ref{lemma_backward_tree}, since by construction they correspond to continuous paths in one-ended components started outside $\backward(2\ell d'+[-k,k]^2)\subset 2\ell d'+(-\ell,\ell)^2$. Hence $\gamma_d\neq\gamma_{d'}$, and since $N_2=1$, then the continuous paths $\pi_{d,1}^+$ and $\pi_{d,2}^+$ must coalesce with $\pi_{d',1}^+$ and $\pi_{d',2}^+$, or $\pi_{d',2}^+$ and $\pi_{d',1}^+$, respectively. One can then verify this yields that $\gamma_d\setminus\gamma_{d'}$ is a portion of $\gamma_d$. Now, define the random set of topological lines 
    \[\Gamma\coloneqq\{\gamma_d:d\in\D\}.\]
By construction, $\Gamma$ is in natural bijection with $\D$ via $d\mapsto\gamma_d$, and satisfies the property that for all $\omega\in\Omega_0$ and $\gamma,\gamma'\in\Gamma(\omega)$, $\gamma\setminus\gamma'$ is a portion of $\gamma$. The following general deterministic theorem ensures that these properties are sufficient to exploit planarity and endow $\Gamma(\omega)$, and hence $\D(\omega)$, with a natural total order. Its proof is deferred to Section~\ref{section_planar_ordering}.

\begin{theoremI}[Planar ordering]
    \label{thm_planar_ordering}
    Let $\Gamma$ be a set of topological lines of the plane, such that for all distinct $\gamma,\gamma'\in\Gamma$, the difference $\gamma\setminus\gamma'$ is a portion of $\gamma$. Then, for all distinct $\gamma,\gamma'\in\Gamma$, denoting by $\triangle$ the symmetric difference operator, the set
        \[J(\gamma,\gamma')\coloneqq\overline{\gamma\triangle\gamma'}\]
    is a Jordan curve. Furthermore, setting
        \[(\gamma,\gamma')_\Gamma\coloneq\{\sigma\in\Gamma\setminus\{\gamma,\gamma'\}:J(\gamma,\sigma)\subset\overline{\mathrm{Int}} J(\gamma,\gamma')\}\]
    for all distinct $\gamma,\gamma'\in\Gamma$, the ternary relation $\cdot\in(\cdot,\cdot)_\Gamma$ induces a \emph{betweenness} structure on $\Gamma$, that is, there exists a strict total order $<_\Gamma$ on $\Gamma$, unique up to reversal, such that for all distinct $\gamma,\gamma',\sigma\in\Gamma$,
        \[\sigma\in(\gamma,\gamma')_\Gamma\iff[\gamma<_\Gamma\sigma<_\Gamma\gamma'\text{ or }\gamma'<_\Gamma\sigma<_\Gamma\gamma].\]
    Finally, for all distinct $\gamma,\gamma',\sigma\in\Gamma$, $\sigma\notin(\gamma,\gamma')_\Gamma\iff J(\gamma,\sigma)\subset\clExt J(\gamma,\gamma')$.
\end{theoremI}

\begin{figure}[ht]
    \label{fig_doors}
    \centering
    \includegraphics[height=0.3\linewidth]{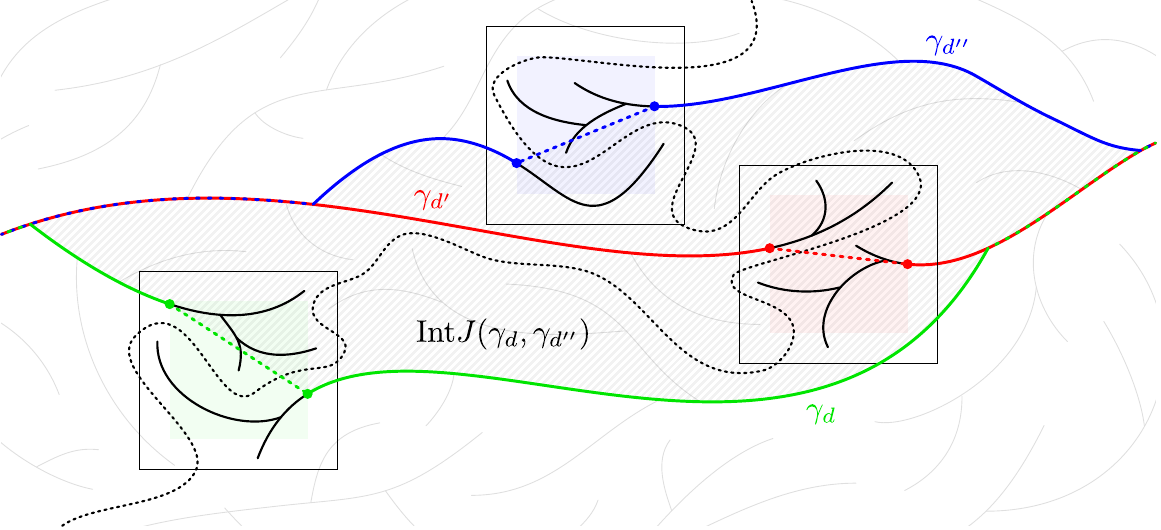}
    \caption{Illustration of the door construction for a fixed realization $\omega\in\Omega_0$ where the doors $d, d', d''\in\mathcal D$ satisfy $d'\in(d, d'')_\Gamma$. Bold black lines indicate $\backward(2\ell z+[-k, k]^2)$ for $z\in\{d, d', d''\}$. The dotted black line represents a two-ended component of $G$. Informally, the red, green, and blue lines delimit the \emph{corridor}, while the corresponding dotted segments serve as the \emph{doors}. That $d'$ lies between $d$ and $d''$ can be seen in the picture from the fact that the red dotted segment representing $d'$ lies in the interior of the corridor section delimited by $\gamma_d$ and $\gamma_{d''}$, hatched in gray.}
\end{figure}

From now on, we denote by $<_\D$ the order on $\D$ induced by $<_\Gamma$ via the bijection $d\mapsto\gamma_d$. Intuitively, for distinct $d,d'\in\D(\omega)$, the Jordan curve $J(\gamma_d,\gamma_{d'})$ may be viewed as the boundary of the corridor section between the corresponding doors. Then, 
    \[(d,d')_\D\coloneqq\{d''\in\D(\omega):\gamma_{d''}(\omega)\in(\gamma_d,\gamma_{d'})_\Gamma\}\]
represents precisely the set of doors lying strictly inside this corridor section. An illustration is provided in Figure~\ref{fig_doors}. One important remark is that for each $\omega\in\Omega_0$, the ordered set
    \[\text{$(\D,<_\D)$ is discrete},\]
in the sense that for any distinct $d,d'\in\D$, $(d,d')_\D$ is finite. Indeed, for all $\omega\in\Omega_0$, if $d,d'\in\D$ are distinct, one can observe that by construction, for each $d''\in\D\setminus\{d,d'\}$, $2\ell d''+[-k,k]^2$ intersects $\gamma_{d''}$ but avoids both $\gamma_d$ and $\gamma_{d'}$. Therefore, if $d''\in(d, d')_\D$, one must have $2\ell d''\in\Int J(\gamma_d, \gamma_{d'})$ which implies that $(d,d')_\D$ is finite since $\#\{k\in\Z^2:2\ell k\in\Int J(\gamma_d, \gamma_{d'})\}<\infty$ as $\Int J(\gamma_d, \gamma_{d'})$ is bounded. 

Before proceeding further, we record the following observation. Although the total order $<_\D$ on $\D$ is only defined up to reversal for every $\omega\in\Omega_0$, and hence arbitrarily chosen, the associated betweenness relation is uniquely determined and measurable, as stated precisely in the lemma below.

\begin{lemma}
    \label{lemma_measurable}
    For all distinct $d,d',d''\in\Z^2$, the map $\omega\in\Omega\mapsto\1_{d\in\D}\1_{d'\in\D}\1_{d''\in(d,d')_\D}$ is measurable.
\end{lemma}

\begin{proof}
    Fix distinct $d,d',d''\in\Z^2$. Measurability of $\omega\mapsto\1_{d\in\D}$ and $\omega\mapsto\1_{d'\in\D}$ is immediate. It therefore suffices to show that, fixing $\omega\in\Omega$ and assuming $d,d'\in\D$, the condition $d''\in(d,d')_\D$, that is
        \[J(\gamma_d,\gamma_{d''})\subset\clInt J(\gamma_d,\gamma_{d'}),\]
    can be expressed in a measurable way. This follows from the fact that membership in the interior $\Int J$ of a Jordan curve $J\subset\R^2$ admits a countable characterization. Indeed, open sets in $\R^2$ are polygonally path connected with intermediate points in $\Q^2$. Hence, for any $z\in\R^2$, one has $z\in\Int J$ if and only if every polygonal path starting at $z$ and ending outside the closed ball of radius $\sup_{y\in J}\|y\|$, with vertices in $\Q^2$ except for $z$, intersects $J$. One can then deduce a countable description of the event of interest, and hence its measurability.
\end{proof}

In what follows, after justifying that all statements can equivalently be formulated in terms of the betweenness relation alone, we will use $<_\D$, which is more intuitive and convenient to handle. We now introduce the key notion of \emph{extreme points}.

\begin{definition}[Extreme points]
    For any subset $X\subset\D$, denote by $\partial_{\D}(X)$ its set of extreme points, that is, elements $x\in X$ which are either a maximum or a minimum in $X$ with respect to $<_\D$. Since this definition is invariant under reversal of the order $<_\D$, it can equivalently be expressed using only the associated betweenness relation.
\end{definition}

Next, we establish in the following lemma that $\D$ admits almost surely no extreme points.

\begin{lemma}
    \label{lemma_no_extremes}
    The set of doors $\D$ has almost surely no extreme points, i.e., $\P[\partial_\D(\D)\neq\emptyset]=0$.
\end{lemma}

\begin{proof}
    First, observe that $\partial_\D(\D)$ is a measurable set using Lemma~\ref{lemma_measurable}. Then, consider the random variable
        \[X\coloneqq\begin{cases}
            \frac1{\#\partial_\D(\D)}\sum_{j\in\partial_\D(\D)}j &\text{if $\partial_\D(\D)\neq\emptyset$,}\\
            \infty &\text{otherwise}.
        \end{cases}\]
    Since $\#\partial_\D[\D(\omega)]\leq 2$ by definition, one has $\partial_\D(\D)\neq\emptyset\iff X\in\frac12\Z^2$. Therefore, observing that $\P(X=0)=\P(X=i)$ for all $i\in\frac12\Z^2$ by \emph{stationarity}, one gets
        \[\P[\partial_\D(\D)\neq\emptyset]=\P(X\in\tfrac12\Z^2)=\sum_{i\in\frac12\Z^2}\P(X=i)=\sum_{i\in\frac12\Z^2}\P(X=0)\in\{0,\infty\},\]
    which gives $\P[\partial_\D(\D)\neq\emptyset]=0$.
\end{proof}

It is worth mentioning that thanks to the previous lemma, for each $\omega\in\Omega_0$ such that $\D\neq\emptyset$, the ordered set $(\D,<_\D)$ is isomorphic to $(\Z,<)$. We show in the following lemma another key feature of $(\D,<_\D)$, namely that any compact region of the plane can be surrounded by the corridor section associated with sufficiently far-apart doors.

\begin{lemma}
    \label{lemma_compact}
    Fix $\omega\in\Omega_0$ such that $\D\neq\emptyset$. For any compact $K\subset\R^2$, there exist distinct $d,d'\in\D$ such that $K\subset\Int J(\gamma_d,\gamma_{d'})$.
\end{lemma}

\begin{proof}
    First, observe that since $N_1=2$, there exists $R>0$ such that
        \[K\subset[-R,R]^2\quad\text{and}\quad N_1([-R,R]^2)=2.\]
    There also exists $L>R$ such that $\backward([-R,R]^2)\subset(-L,L)^2$. Define
        \[\D_0\coloneqq\{d\in\D:[2\ell d+(-\ell,\ell)^2]\cap(-L,L)^2\neq\emptyset\}.\]
    By construction, $\D_0$ is finite. Let $u_1^*$ and $u_2^*$ be arbitrary chosen vertices in $[-R,R]^2$ that belong to different one-ended components of $G$. Set $\pi_1^*\coloneq\pi_{u_1^*}[\cdot-t_{u_1^*}^+([-R,R]^2)]$ for each $i\in\{1,2\}$ and denote
        \[\gamma^*\coloneqq \pi_1^*(\R_+)\cup[\pi_1^*(0),\pi_2^*(0)]\cup\pi_2^*(\R_+).\]
    Consider the family
        \[\Gamma'\coloneqq(\Gamma\setminus\{\gamma_d:d\in\D_0\})\cup\{\gamma^*\}.\]
    By construction, $\gamma^*$ is a topological line, and the same argument that ensured Theorem~\ref{thm_planar_ordering} applies to $\Gamma$ shows it also applies to $\Gamma'$. Then, by the same argument that established the discreteness of $(\D,<_\D)$ is discrete, one checks that $(\Gamma',<_{\Gamma'})$ is also discrete. Furthermore, $(\Gamma,<_{\Gamma})$ is non-empty and has no extreme points from $\D(\omega)\neq\emptyset$ and $\partial_\D(\D)=\emptyset$. Hence, as the two ordered sets are compatible and differ only by finitely many elements, namely $\{\gamma^*\}\cup\{\gamma_d:d\in\D_0\}$, it follows that $(\Gamma',<_{\Gamma'})$ admits no extreme point. In particular, $\gamma^*$ is not extreme, and thus there exist distinct $d,d'\in\D\setminus\D_0$ such that
        \[J(\gamma_d,\gamma^*)\subset\clInt J(\gamma_d,\gamma_{d'}).\]
    Finally, observing that by construction, $(-L,L)^2$ contains $K$, intersects $\gamma^*$, but avoids both $\gamma_d$ and $\gamma_{d'}$, one concludes that $K\subset\Int J(\gamma_d,\gamma_{d'})$.
\end{proof}

Now that the corridor structure is well-established, we make the link between doors and two-ended components via the following definition.

\begin{definition}[Door trace]
    Let $\mathcal C_2$ denotes the set of two-ended components of $G$ as connected subsets of $\R^2$ obtained by taking the union of their edges as line segments. Then, for each $C\in\mathcal C_2$, define its \emph{door trace} by
        \[\trace(C)\coloneqq\{d\in\D:C\cap\gamma_d\neq\emptyset\}.\]
\end{definition}

Note importantly that thanks to the \emph{planarity} assumption, for all $\omega\in\Omega_0$ and $d\in\D'$, any two-ended component $C\in\mathcal C_2$ can only intersect $\gamma_d$ within the line segment 
    \[[\pi_{d,1}^+(0),\pi_{d,2}^+(0)],\]
which avoids $\gamma_{d'}$ for all $d'\in\D\setminus\{d\}$ by construction. The following proposition formalizes the intuitive idea that door traces are unbounded intervals.

\begin{proposition}
    \label{prop_unbounded_convex}
    Fix $\omega\in\Omega_0$ such that $\D\neq\emptyset$. Then, for each $C\in\mathcal C_2$, $\trace(C)$ is unbounded and convex, that is for every distinct $d,d'\in\trace(C)$, $(d,d')_{\D}\subset\trace(C)$.
\end{proposition}

\begin{proof}
    Let $C\in\mathcal C_2$. First, observe that choosing a non-empty compact subset $K\subset C$, applying Lemma~\ref{lemma_compact} gives that there exist distinct $d,d'\in\D$ such that
        \[K\subset \Int J(\gamma_d,\gamma_{d'}).\]
    Then, since $C$ is connected and unbounded and contains $K$, it implies that $C\cap J(\gamma_d,\gamma_{d'})\neq\emptyset$ and either $d$ or $d'$ must belong to $\trace(C)$. This shows that $\trace(C)$ is non-empty. Now let us show that $\trace(C)$ is unbounded. Assume by contradiction that there exists distinct $d,d'\in\D$ such that $\trace(C)\subset(d, d')_\D$. Since $\trace(C)\neq\emptyset$, one can fix $d''\in\trace(C)\subset(d, d')_\D$. Then, $C$ intersects
        \[\gamma_{d''}\setminus(\gamma_d\cup \gamma_{d'})\subset J(\gamma_d,\gamma_{d''})\setminus J(\gamma_d,\gamma_{d'})\subset\Int J(\gamma_d,\gamma_{d'}),\]
    where the inclusion uses $J(\gamma_d,\gamma_{d''})\subset\clInt J(\gamma_d,\gamma_{d'}()$ from $d''\in(d,d')_\D$. But then, since $C$ is unbounded and connected, it must cross $J(\gamma_d,\gamma_{d'})$, which is absurd since it avoids both $\gamma_d$ and $\gamma_{d'}$ from $\trace(C)\subset(d, d')_\D$. It remains to check convexity. Fix distinct $d, d'\in\trace(C)$ and $d''\in(d,d')_\D$. Without loss of generality, one can assume $d<_\D d'$. Since $d,d'\in\trace(C)$, there exist
        \[z_0\in C\cap\gamma_d\quad\text{and}\quad z_1\in C\cap\gamma_{d'}.\]
    Now, because $C$ is path-connected, there also exists a continuous function $f:[0,1]\to C$ with $f(0)=z_0$ and $f(1)=z_1$. Since $f([0,1])$ is compact, one can apply Lemma~\ref{lemma_compact} to obtain the existence of $d^-,d^+\in\D$ with $d^-<_\D d^+$ such that
        \[f([0,1])\subset\Int J(\gamma_{d^-}, \gamma_{d^+}).\]
    Since $z_0\in f([0,1])\cap J(\gamma_{d^-}, \gamma_{d})$, this shows $J(\gamma_{d^-},\gamma_d)\not\subset\clExt J(\gamma_{d^-},\gamma_{d^+})$ and the last assertion of Theorem~\ref{thm_planar_ordering} gives $d\in(d^-,d^+)_\D$. Then, by transitivity, one gets
        \[d^-<_\D d<_\D d''<_\D d'.\]
    This has two consequences. First,
        \[z_0\in\gamma_d\setminus(\gamma_{d^-}\cup\gamma_{d''})\subset J(\gamma_{d^-},\gamma_{d})\setminus J(\gamma_{d^-},\gamma_{d''})\subset\Int J(\gamma_{d^-},\gamma_{d''})\]
    from $d\in(d^-,d'')_\D$. Second, one must have
        \[z_1\in\gamma_{d'}\setminus\gamma_{d^-}\subset J(\gamma_{d^-},\gamma_{d'})\not\subset\clExt J(\gamma_{d^-},\gamma_{d''})\]
    using the last assertion of Theorem~\ref{thm_planar_ordering} with $d'\notin(d^-,d'')_\D$. Together, this implies that there exists $t\in[0,1]$ such that
        \[f(t)\in J(\gamma_{d^-},\gamma_{d''}).\]
    Then, since $f(t)$ must lie in $C\cap\Int J(\gamma_{d^-}, \gamma_{d^+})$ and hence must belong to $\gamma_{d^-}$, one has necessarily
        \[f(t)\in\gamma_{d''}\cap C,\]
    which shows $d''\in\trace(C)$. This completes the proof.
\end{proof}

Let us now assume without loss of generality from the \emph{finite edge intensity} that for all $\omega\in\Omega_0$ and compact $K\subset\R^2$, $\#\{(u, v)\in E:[u, v]\cap K\neq\emptyset\}<\infty$. Then since any component $C\in\mathcal C_2$ that cross a door $d\in\D$ must intersect the compact set $2\ell d+[-k, k]^2$, so that for each $\omega\in\Omega_0$, doors cross only finitely many two-ended components, i.e.,
    \[\#\{C\in\mathcal C_2:d\in\trace(C)\}<\infty\]
for all $d\in\D$. The last step toward the proof that $N_2$ must be finite is the following proposition.

\begin{proposition}
    \label{prop_cross_all_doors}
    Conditional on $\{\D\neq\emptyset\}$, all two-ended components cross every door with probability one, that is
        \[\P[\forall C\in\mathcal C_2,~\trace(C)=\D~|~\D\neq\emptyset]=1.\]
\end{proposition}

\begin{proof}
    For each $\omega\in\Omega_0$, let $M^-$ and $M^+$ denote the subset of $\D$ consisting of those
    $d\in\D$ that are minimal and maximal, respectively, in $\trace(C)$ for some $C\in\mathcal C_2$, and define $Y$ as the set of $d\in\D$ that are either minimal in $M^-$ or maximal in $M^+$. First, observe that since the definition of $Y$ is invariant under reversal of $<_\D$, it admits a characterization using only the associated betweenness relation, so that one can define the random variable
        \[X\coloneqq\begin{cases}
            \frac1{\#Y}\sum_{j\in Y}j&\text{if $Y\neq\emptyset$},\\
            \infty&\text{otherwise}.
        \end{cases}\]
    Now, fix $\omega\in\Omega_0$ such that $\D\neq\emptyset$. Observe that for each $d\in\D$, since $\trace(C)$ is unbounded for each $C\in\mathcal C_2$ from Proposition~\ref{prop_unbounded_convex}, then the sets
        \[\{d'\in M^-:d'<_\D d\}\quad\text{and}\quad\{d'\in M^+:d<_\D d'\}\]
    correspond to minimal and maximal doors of distinct two-ended components crossing $d$, and hence must be both finite. This shows that $M^-$ is lower bounded and $M^+$ is upper bounded and thus $M^-\cup M^+\neq\emptyset\iff X\neq\emptyset$ since $(\D,<_\D)$ is discrete. Additionally, for each $C\in\mathcal C_2$, since $(\D,<_\D)$ is discrete and $\trace(C)$ is convex for all $C\in\mathcal C_2$ from Proposition~\ref{prop_unbounded_convex}, then $\trace(C)$ differs from $\D$ if and only if it admits a maximum or a minimum. Together, this yields that for all $\omega\in\Omega_0$ such that $\D\neq\emptyset$,
        \[[\exists C\in\mathcal C_2,~\trace(C)\neq\D]\iff M^-\cup M^+\neq\emptyset\iff Y\neq\emptyset\iff X\in\tfrac12\Z^2.\]
    Finally, observing that by \emph{stationarity}, $\P(X=i~|~\D\neq\emptyset)=\P(X=0~|~\D\neq\emptyset)$ for all $i\in\frac12\Z^2$, one can write
        \begin{equation*}
            \begin{split}
                \P[\exists C\in\mathcal C_2,~\trace(C)\neq\D~|~\D\neq\emptyset]&=\P(X\in\tfrac12\Z^2~|~\D\neq\emptyset)\\
                &=\sum_{i\in\frac12\Z^2}\P(X=i~|~\D\neq\emptyset)
                =\sum_{i\in\frac12\Z^2}\P(X=0~|~\D\neq\emptyset)\in\{0,\infty\},
            \end{split}
        \end{equation*}
    which yields the result.
\end{proof}

We finally prove the main result of the section.

\begin{proof}[Proof of Proposition~\ref{prop_N1_is_2_implies_N2_finite}]
    From Proposition~\ref{prop_cross_all_doors}, conditional on $\{\D\neq\emptyset\}$, all two-ended components of $G$ must cross every door. Then, since every door can only cross finitely many components, this yields
        \[\P(N_2<\infty~|~\D\neq\emptyset)=1,\]
    and thus $N_2<\infty$ almost surely since $N_2$ is constant on an event of full-probability.
\end{proof}

Theorem~\ref{thm_main_undirected} then follows by combining Propositions~\ref{prop_at_most_2_ends}, \ref{prop_N1_leq_2}, and \ref{prop_N1_is_2_implies_N2_finite}, the examples constructed in Section~\ref{section_examples}, and the following straightforward lemma, which generalizes directly to higher dimensions.

\begin{lemma}
    Almost surely, $N_0\in\{0,\infty\}$.
\end{lemma}

\begin{proof}
    If $N_0\geq1$ almost surely, then there exists $\ell\geq1$ such that the box $[-\ell,\ell)^2$ contains all the vertices of a finite component of $G$ with positive probability. Then, denoting by $A_\ell$ this event, one has $N_0\geq\sum_{i\in\Z^2}\1_{\theta_{-2\ell i}(\omega)\in A_\ell}$ almost surely. Finally, taking the expectation using \emph{stationarity}, it comes $\E[N_0]\geq\sum_{i\in\Z^2}\P(A_\ell)=\infty$, which implies $N_0=\infty$ a.s.\ since $N_0$ is constant on a full-probability event. This shows that $N_0\in\{0,\infty\}$ almost surely.
\end{proof}

\section{Proof of the planar ordering theorem}

\label{section_planar_ordering}

This section is devoted to proving Theorem~\ref{thm_planar_ordering}. In what follows, we fix the deterministic set $\Gamma$ of topological lines such that $\gamma\setminus\gamma'$ is a portion of $\gamma$ for all distinct $\gamma,\gamma'\in\Gamma$. We begin with the following lemma that provides useful adapted parametrizations of the topological lines.

\begin{lemma}
    \label{lemma_adapted_homeo}
    Fix $\gamma\in\Gamma$ and an homeomorphism $\Phi:\R\to\gamma$. For each $\sigma\in\Gamma\setminus\{\gamma\}$, there exists a homeomorphism $\phi_{\sigma}$ and a non-empty bounded open interval $I_{\sigma}\coloneqq (a_\sigma,b_\sigma)\subset\R$ such that
        \[\text{$\Phi=\phi_\sigma$ on $\R\setminus I_\sigma$},\quad \phi_\sigma(I_\sigma)=\sigma\setminus\gamma,\quad\text{and}\quad \Phi(I_\sigma)=\gamma\setminus\sigma.\]
    Then, one has additionally that for all distinct $\sigma,\sigma'\in\Gamma\setminus\{\gamma\}$,
        \[(\partial I_{\sigma'})\setminus\overline{I_{\sigma}}\subset \overline{I_{\sigma'|\sigma}}\quad\text{where}\quad I_{\sigma'|\sigma}\coloneqq\phi_{\sigma'}^{-1}(\sigma'\setminus\sigma).\]
\end{lemma}

\begin{proof}
    Let $\sigma\in\Gamma\setminus\{\gamma\}$ and fix an homeomorphism $\psi_\sigma:\R\to$. First, observe that 
        \[\gamma\setminus\sigma\neq\emptyset\]
    Indeed, assuming $\gamma\setminus\sigma=\emptyset$ by contradiction, one has $\gamma\subset\sigma$ and since $\gamma$ is connected, distinct from $\sigma$, $\psi_\sigma^{-1}(\gamma)$ must be a non-empty proper interval of $\R$, so that $\psi_\sigma^{-1}(\sigma\setminus\gamma)=\R\setminus\psi^{-1}(\gamma)$ is unbounded, which is absurd since $\sigma\setminus\gamma$ is a portion of $\sigma$. Therefore, $\gamma\setminus\sigma$ is a non-empty portion of $\gamma$, i.e., there exist $a_{\sigma}<b_{\sigma}$ in $\R$ such that
        \[\Phi^{-1}(\gamma\setminus\sigma)=I_{\sigma}\coloneqq(a_{\sigma},b_{\sigma}).\]
    Additionally, since $\sigma\setminus\gamma$ is a portion of $\sigma$, there exists $a_\sigma^*\leq b_{\sigma}^*$ in $\R$ such that
        \[\psi_{\sigma}^{-1}(\sigma\setminus\gamma)= I_{\sigma}^*\coloneq( a_{\sigma}^*, b_{\sigma}^*).\]
    Now, consider the map
        \[f_\sigma:\R\setminus I_{\sigma}\to\R\setminus I_\sigma^*,~t\mapsto\psi_{\sigma}^{-1}\circ\Phi(t).\]
    By construction, it is a well-defined homeomorphism. In particular, since $\R\setminus I_\sigma$ is not connected, $I_\sigma^*$ must be non-empty, i.e., $a_\sigma^*<b_\sigma^*$. Then, identifying the connected components of $\R\setminus I_\sigma$ and $\R\setminus I_\sigma^*$, $f_\sigma$ restricts to monotonous homeomorphisms from $(-\infty,a_\sigma]$ and $[b_\sigma,\infty)$ to either $(-\infty,a_\sigma^*]$ and $[b_\sigma^*,\infty)$, or $[b_\sigma^*,\infty)$ and $(-\infty,a_\sigma^*]$, respectively. Therefore, identifying the unique extrema of each component, one deduces
        \[\{f_\sigma(a_\sigma),f_\sigma(b_\sigma)\}=\{a_\sigma^*,b_\sigma^*\}.\]
    Hence, considering the unique affine function $L_\sigma:\R\to\R$ that maps $a_\sigma$ onto $f_\sigma(a_\sigma)$ and $b_\sigma$ onto $f_\sigma(b_\sigma)$, one can extend $f_\sigma$ by setting
        \[h_\sigma:\R\to\R,~t\mapsto\begin{cases}
            f_\sigma(t)&\text{if $t\notin I_\sigma$,}\\
            L_\sigma(t)&\text{otherwise}.
        \end{cases}\]
    By construction, $h_\sigma$ is an homeomorphism with $h(I_\sigma)=I_\sigma^*$ and thus, the map
        \[\phi_{\sigma}\coloneqq\psi_\sigma\circ h_\sigma\]
    realizes an homeomorphism from $\R$ to $\sigma$ such that $\Phi=\phi_\sigma$ on $\R\setminus I_\sigma$, $\phi_\sigma(I_\sigma)=\sigma\setminus\gamma$ and $\Phi(I_\sigma)=\gamma\setminus\sigma$.
    
    It remains to address the second assertion of the lemma. Fix distinct $\sigma,\sigma'\in\Gamma\setminus\{\gamma\}$ and let $t\in(\partial I_{\sigma'})\setminus\overline{I_{\sigma}}$. Since $t\notin \overline{I_\sigma}$, $\overline{I_\sigma}$ is compact and $\phi_\sigma$ is injective,
        \[\rho\coloneqq\inf_{t'\in I_\sigma}\|\phi_\sigma(t')-\phi_\sigma(t)\|>0.\]
    Additionally, since $\phi_{\sigma'}$ is continuous at $t$, there exists $\delta>0$ such that for all $t'\in\R$
        \[|t'-t|<\delta\implies\|\phi_{\sigma'}(t')-\phi_{\sigma'}(t)\|<\rho.\]
    Now fix $\varepsilon\in(0,\delta)$. Since $t\in\partial I_{\sigma'}$, there exists $t'\in I_{\sigma'}$ with $|t'-t|<\varepsilon$. Then, since $t\in\R\setminus(I_\sigma\cup I_{\sigma'})$, one has $\phi_\sigma(t)=\Phi(t)=\phi_{\sigma'}(t)$ and thus
        \[\|\phi_{\sigma'}(t')-\phi_{\sigma}(t)\|=\|\phi_{\sigma'}(t')-\phi_{\sigma'}(t)\|<\rho\]
    which, by definition of $\rho$, implies $\phi_{\sigma'}(t')\notin\phi_\sigma(I_\sigma)$. Then, since one has also $\phi_{\sigma'}(t')\notin\gamma$ from $t'\in I_{\sigma'}$, it comes that $\phi_{\sigma'}(t')\notin\phi_\sigma(\R)=\sigma$, i.e.,
        \[t'\in\phi_{\sigma'}^{-1}(\sigma'\setminus\sigma)=I_{\sigma'|\sigma}.\]
    Since $\varepsilon$ was arbitrary, it implies that $t\in\overline{I_{\sigma'|\sigma}}$. This shows $(\partial I_{\sigma'})\setminus\overline{I_{\sigma}}\subset \overline{I_{\sigma'|\sigma}}$.
\end{proof}

Next, we verify that the elements of $\Gamma$ indeed allow us to construct Jordan curves, which ensures $\cdot\in(\cdot,\cdot)_\Gamma$ is well defined, and establish some key properties.

\begin{proposition}
    \label{prop_arc_diff}
    For all distinct $\gamma,\gamma'\in\Gamma$,
        \[J(\gamma,\gamma')\coloneqq\overline{\gamma\triangle\gamma'}\]
    is a Jordan curve. Furthermore, for any $\gamma\in\Gamma$ and non-empty finite subset $\Gamma_0\subset\Gamma\setminus\{\gamma\}$,
        \[\gamma\setminus\bigcup_{\sigma\in\Gamma_0}\sigma\]
    is a non-empty portion of $\gamma$, and for all distinct $\gamma,\sigma,\sigma'\in\Gamma$, the difference $J(\gamma,\sigma')\setminus J(\gamma,\sigma)$ is a connected subset of $J(\sigma,\sigma')$.
\end{proposition}

\begin{proof}
    Fix $\gamma\in\Gamma$ and an homeomorphism $\Phi:\R\to\gamma$. Let 
        \[(\phi_\sigma)_{\sigma\in\Gamma\setminus\{\gamma\}}\quad\text{and}\quad(I_\sigma)_{\sigma\in\Gamma\setminus\{\gamma\}}=[(a_\sigma,b_\sigma)]_{\sigma\in\Gamma\setminus\{\gamma\}}\]
    be the homeomorphisms and non-empty bounded open intervals given by Lemma~\ref{lemma_adapted_homeo}. Let $\gamma'\in\Gamma\setminus\{\gamma\}$. Since by construction $\Phi(I_{\gamma'})=\gamma\setminus\gamma'$ and $\phi_{\gamma'}(I_{\gamma'})=\gamma'\setminus\gamma$ are disjoint, and $\Phi=\phi_{\gamma'}$ on the two distinct points of $\partial I_{\gamma'}$, then
        \[J(\gamma,\gamma')=\Phi(\overline{I_{\gamma'}})\cup\phi_{\gamma'}(\overline{I_{\gamma'}})\]
    is a Jordan curve.

    Now, let us fix a non-empty finite subset $\Gamma_0\subset\Gamma\setminus\{\gamma\}$ and show that $\gamma\setminus\bigcup_{\sigma\in\Gamma_0}\sigma$ is a non-empty portion of $\gamma$. By construction,
        \[\gamma\setminus\bigcup_{\sigma\in\Gamma_0}\sigma=\Phi(I_{\Gamma_0})\quad\text{where}\quad I_{\Gamma_0}\coloneqq\bigcap_{\sigma\in\Gamma_0}I_\sigma.\]
    Let $\sigma,\sigma'\in\Gamma_0$ be such that $b_{\sigma'}=\max_{\nu\in\Gamma_0} b_\nu$ and $a_{\sigma}=\min_{\nu\in\Gamma_0} a_\nu$. Then, $a_{\sigma}\leq b_{\sigma'}$ and
        \[I_{\Gamma_0}=(a_{\sigma}, b_{\sigma'}).\]
    It suffices therefore to show that $a_{\sigma}<b_{\sigma'}$. Assume by contradiction that $a_{\sigma}=b_{\sigma'}$. Then $I_{\sigma}$ and $I_{\sigma'}$ must be disjoint. Fix $t\in I_{\sigma}$. By definition, $\Phi(t)\notin\sigma$, and $\phi_{\sigma'}(t)=\Phi(t)$ since $t\notin I_{\sigma'}$. This gives $\phi_{\sigma'}(t)\in\sigma'\setminus\sigma$ and hence
        \[t\in I_{\sigma'|\sigma}.\]
    Additionally, both $a_{\sigma}$ and $b_{\sigma}$ belong to $\R\setminus(I_{\sigma}\cup I_{\sigma'})$, so that $\phi_{\sigma'}(a_{\sigma})=\Phi(a_{\sigma})=\phi_{\sigma}(a_{\sigma})\in\sigma$ and $\phi_{\sigma'}(b_{\sigma})=\Phi(b_{\sigma})=\phi_{\sigma}(b_{\sigma})\in\sigma$, which implies that
        \[a_{\sigma}\notin I_{\sigma'|\sigma}\quad\text{and}\quad b_{\sigma}\notin I_{\sigma'|\sigma}.\]
    Together, it shows that $I_{\sigma'|\sigma}$ must contain $t\in(a_\sigma,b_\sigma)$ but not $a_\sigma$ nor $b_\sigma$. Since $\sigma'\setminus\sigma$ is a portion of $\sigma'$, $I_{\sigma'|\sigma}=\phi_{\sigma'}^{-1}(\sigma'\setminus\sigma)$ must be an open interval, thus one must have
        \[I_{\sigma'|\sigma}\subset I_\sigma.\]
    But then, applying Lemma~\ref{lemma_adapted_homeo}, it comes $(\partial I_{\sigma'})\setminus\overline{I_\sigma}\subset\overline{I_{\sigma'|\sigma}}\subset\overline{I_\sigma}$, i.e., $\partial I_{\sigma'}\subset\overline{I_\sigma}$, which is absurd since $b_{\sigma'}\in\partial I_{\sigma'}$ but $b_{\sigma'}\notin\partial I_\sigma$ as $b_{\sigma'}>a_{\sigma'}\geq b_\sigma$.

    It now remains to prove the last assertion. Fix distinct $\sigma,\sigma'\in\Gamma\setminus\{\gamma\}$. Let us prove that $J(\gamma,\sigma')\setminus J(\gamma,\sigma)$ is a connected subset of $J(\sigma,\sigma')$. First, observe that by construction, one can write
        \[J(\gamma,\sigma')=\Phi(\overline{I_{\sigma'}})\cup \phi_{\sigma'}(I_{\sigma'})\quad\text{and}\quad J(\gamma,\sigma)=\Phi(\overline{I_{\sigma}})\cup \phi_{\sigma}(I_{\sigma}).\]
    Then, taking the difference, it comes
        \begin{equation}
            \label{eq_Jordan_diff}
            \begin{split}
                J(\gamma,\sigma')\setminus J(\gamma,\sigma)&=\Phi(\overline{I_{\sigma'}})\setminus[\Phi(\overline{I_{\sigma}})\cup \phi_{\sigma}(I_{\sigma})]\cup\phi_{\sigma'}(I_{\sigma'})\setminus[\Phi(\overline{I_\sigma})\cup\phi_{\sigma}(I_{\sigma})]\\
                &=\Phi(\overline{{I_{\sigma'}}}\setminus\overline{I_\sigma})\cup\phi_{\sigma'}(I_{\sigma'})\setminus\phi_{\sigma}(I_{\sigma})\\
                &=\Phi(\overline{{I_{\sigma'}}}\setminus\overline{I_\sigma})\cup\phi_{\sigma'}(I_{\sigma'}\cap I_{\sigma'|\sigma})
            \end{split}
        \end{equation}
    where using the definition of $I_\sigma$ and $I_{\sigma'}$, the first equality follows from $\phi_\sigma(I_\sigma)\cap\Phi(\overline{I_{\sigma'}})\subset\phi_\sigma(I_\sigma)\cap\gamma=\emptyset$ as well as $\phi_{\sigma'}(I_{\sigma'})\cap\Phi(\overline{I_\sigma})\subset\phi_{\sigma'}(I_{\sigma'})\cap\gamma=\emptyset$, and $\phi_\sigma(\R\setminus I_{\sigma})\cap\phi_{\sigma'}(I_{\sigma'})\subset\gamma\cap\phi_{\sigma'}(I_{\sigma'})=\emptyset$ yields the last line. Now, by construction, $\phi_{\sigma'}(I_{\sigma'}\cap I_{\sigma'|\sigma})\subset\phi_{\sigma'}(I_{\sigma'|\sigma})=\sigma'\setminus\sigma$, and
        \[\Phi(\overline{{I_{\sigma'}}}\setminus\overline{I_\sigma})\subset\overline{\Phi(I_{\sigma'})\setminus\Phi(I_\sigma)}\subset\overline{(\gamma\setminus\sigma')\setminus(\gamma\setminus\sigma)}\subset\overline{\sigma\setminus\sigma'}.\]
    Injected in~\eqref{eq_Jordan_diff}, it shows
        \[J(\gamma,\sigma')\setminus J(\gamma,\sigma)\subset\overline{\sigma\setminus\sigma'}\cup\sigma\setminus\sigma'\subset \overline{\sigma\triangle\sigma'}=J(\sigma,\sigma').\]
    It remains to check connectedness. To that end, observe that since $I_{\sigma'}\cap I_{\sigma'|\sigma}$ is an open interval,
        \[\phi_{\sigma'}(I_{\sigma'}\cap I_{\sigma'|\sigma})\]
    is connected. Then, observe that although $\overline{{I_{\sigma'}}}\setminus\overline{I_\sigma}$ may be disconnected, each of its non-empty components must contain a point of $(\partial I_{\sigma'})\setminus\overline{I_\sigma}$. Finally, since
        \[(\partial I_{\sigma'})\setminus\overline{I_\sigma}\subset\overline{I_{\sigma'|\sigma}}\cap \overline{I_{\sigma'}}=\overline{I_{\sigma'}\cap I_{\sigma'|\sigma}}\]
    from Lemma~\ref{lemma_adapted_homeo}, each component of $\Phi(\overline{{I_{\sigma'}}}\setminus\overline{I_\sigma})$ connects to $\phi_{\sigma'}(I_{\sigma'}\cap I_{\sigma'|\sigma})$, so that injecting in~\eqref{eq_Jordan_diff}, it comes $J(\gamma,\sigma')\setminus J(\gamma,\sigma)$ is connected.
\end{proof}

It remains to check that $\cdot\in(\cdot,\cdot)_\Gamma$ satisfies the axioms of a betweenness relation given by \cite{betweenness}, namely symmetry (Corollary~\ref{coro_containment_equiv}) as well as transitivity and the trichotomy property (Proposition~\ref{prop_transitive_and_trichotomy}). The following is a consequence of the previous proposition and already establishes symmetry.

\begin{corollary}
    \label{coro_containment_equiv}
    For all distinct $\gamma,\gamma',\sigma\in\Gamma$, one has either
        \[J(\gamma,\sigma)\subset\clInt J(\gamma, \gamma')\quad\text{or}\quad J(\gamma,\sigma)\subset\clExt J(\gamma,\gamma').\]
    Furthermore, $(\cdot,\cdot)_{\Gamma}$ is \emph{symmetric}, i.e., for all distinct $\gamma,\gamma'\in\Gamma$, $(\gamma,\gamma')_\Gamma=(\gamma',\gamma)_\Gamma$.
\end{corollary}

\begin{proof}
    Let $\gamma,\gamma',\sigma\in\Gamma$ be distinct. Since $J(\gamma,\sigma)\setminus J(\gamma,\gamma')$ avoids $J(\gamma,\gamma')$, and is connected by Proposition~\ref{prop_arc_diff}, one has either
        \[J(\gamma,\sigma)\setminus J(\gamma,\gamma')\subset \Int J(\gamma,\gamma')\quad\text{or}\quad J(\gamma,\sigma)\setminus J(\gamma,\gamma')\subset\Ext J(\gamma,\gamma').\]
    Then, taking the union with $J(\gamma,\sigma)\cap J(\gamma,\gamma')\subset J(\gamma,\gamma')=\clInt J(\gamma,\gamma')\cap \clExt J(\gamma,\gamma')$, the first assertion follows. Then, one can observe that since
        \[\emptyset\neq\sigma\setminus(\gamma\cup\gamma')\subset J(\gamma,\sigma)\setminus J(\gamma,\gamma')\]
    from Proposition~\ref{prop_arc_diff} and $J(\gamma,\gamma')=J(\gamma',\gamma)$, one can write
        \begin{equation*}
            \begin{split}
                \sigma\in(\gamma,\gamma')_\Gamma&\iff J(\gamma,\sigma)\subset\clInt J(\gamma, \gamma')\\
                &\iff \sigma\setminus (\gamma\cup\gamma')\subset\Int J(\gamma,\gamma')\\
                &\iff \sigma\setminus (\gamma'\cup\gamma)\subset\Int J(\gamma',\gamma)\\
                &\iff J(\gamma',\sigma)\subset\clInt J(\gamma', \gamma)
                \iff\sigma\in(\gamma',\gamma)_\Gamma,
            \end{split}
        \end{equation*}
    which proves symmetry.
\end{proof}

Now, in order to tackle the geometry of interacting Jordan curves and to obtain the remaining sought properties, we will need the two following technical lemmas.

\begin{lemma}
    \label{lemma_inclusions_deduction}
    For any Jordan curves $J,J'\subset\R^2$, one has $J\subset\clInt J'\implies\Int J\subset\Int J'$, and
        \[[J\subset\clExt J'\text{ and }J'\not\subset\clInt J]\implies\Int J\subset\Ext J'.\]
\end{lemma}

\begin{proof}
    Fix $z\in\clInt J$ and consider an unbounded continuous path $f:\R_+\to\R^2$ with $f(0)=z$. Since $\clInt J$ is bounded, there exists $t_0\in\R_+$ such that $f(t_0)\in\partial\clInt J=J$. Then, as $J\subset\clInt J'$ and $\clInt J'$ is unbounded, there must also exist $t_1\geq t_0$ such that $f(t_1)\in\partial\clInt J'=J'$. This shows that all unbounded continuous path started at $z$ eventually cross $J'$, i.e., $z\in\clInt J'$. Then, one has $\clInt J\subset\clInt J'$ and finally, removing the boundaries, it comes $\Int J\subset\Int J'$.
    
    Now remains to check that $[J\subset\clExt J'\text{ and }J'\not\subset\clInt J]\implies\Int J\subset\Ext J'$. To do so, observe that it is sufficient by contrapositive to show $[\Int J\not\subset\Ext J'\text{ and }J'\not\subset\clInt J]\implies J\not\subset\clExt J'$. Assume that $\Int J\not\subset\Ext J$ and $J'\not\subset\clInt J$. From $\Int J\not\subset\Ext J'$, one has $\Int J\not\subset\clExt J'$ and there must exist
        \[z_0\in\Int J\cap\Int J'.\]
    Additionally, one has $\Int J'\not\subset\Int J$ from $\partial\Int J'=J'\not\subset\clInt J$, and there must exist
        \[z_1\in\Int J\setminus\Int J'.\]
    Then, as $z_0$ and $z_1$ both belong to $\Int J$ which is path-connected, there exists a continuous path $g:[0,1]\to\Int J$ with $g(0)=z_0$ and $g(1)=z_1$. Since $g(0)\in\Int J'$ but $g(1)\notin\Int J'$, and $g$ is continuous, there exists $t\in[0,1]$ such that $g(t)\in\partial\Int J'=J'$. Finally, as $g(t)\in\Int J'$ by construction of $g$, it shows $g(t)\in J'\cap\Int J$ and thus, $J'\not\subset\clExt J$. The proof is complete.
\end{proof}

\begin{lemma}
    \label{lemma_local_identification}
    For any distinct $\gamma,\sigma,\sigma'\in\Gamma$ and $z\in J(\gamma,\sigma')\setminus J(\gamma,\sigma)$, there exists an open set $U\subset\R^2$ containing $z$ such that
        \[\{U\cap\Int J(\gamma,\sigma'), U\cap\Ext J(\gamma,\sigma')\}=\{U\cap\Int J(\sigma,\sigma'), U\cap\Ext J(\sigma,\sigma')\}\]
\end{lemma}

\begin{proof}
    Let $\gamma,\sigma,\sigma'\in\Gamma$. Fix $z\in J(\gamma,\sigma')\setminus J(\gamma,\sigma)$. Applying the Jordan-Schönflies theorem (see~\cite{topology_book}), there exists a homeomorphism
        \[\Psi:\R^2\to\R^2\]
    such that $\Psi[J(\gamma,\sigma')]=\mathbb S^1$. Since $z\in\R^2\setminus J(\gamma,\sigma)$ and $\Psi[\R^2\setminus J(\gamma,\sigma')]$ is open, there exist $\varepsilon>0$ small enough so that the open disk
        \[D\coloneqq\{y\in\R^2:\|y-\Psi(x)\|<\varepsilon\}\]
    is included in $\Psi[\R^2\setminus J(\gamma,\sigma)]$. Clearly,
        \[D\setminus\mathbb S^1\]
    has exactly two connected components. Now set $U\coloneqq\Psi^{-1}(D)$. Then,
        \[U\setminus J(\gamma,\sigma')=\Psi^{-1}(D\setminus\mathbb S^1)\]
    must consist of two components, which are necessarily $U\cap\Int J(\gamma,\sigma')$ and $U\cap\Ext J(\gamma,\sigma')$. Now, observe that
        \begin{equation*}
            \begin{split}
                [U\setminus J(\gamma,\sigma')]\triangle[U\setminus J(\sigma,\sigma')]&=U\cap[J(\gamma,\sigma')\triangle J(\sigma,\sigma')]\\
                &\subset U\cap J(\gamma,\sigma)\\
                &\subset[\R^2\setminus J(\gamma,\sigma)]\cap J(\gamma,\sigma)\\
                &=\emptyset,
            \end{split}
        \end{equation*}
    where the first inclusion uses that both $J(\gamma,\sigma')\setminus J(\sigma,\sigma')$ and $J(\gamma,\sigma')\setminus J(\sigma,\sigma')$ are subsets of $J(\gamma,\sigma)$ from Proposition~\ref{prop_arc_diff}. This shows $U\setminus J(\sigma,\sigma')=U\setminus J(\gamma,\sigma')$, so that the two components of those sets can be alternatively expressed as $U\cap\Int J(\sigma,\sigma')$ and $U\cap\Ext J(\sigma,\sigma')$, and the result follows by identifying.
\end{proof}

The last key ingredient for establishing the transitive property of $\cdot\in(\cdot,\cdot)_\Gamma$ is provided by the following proposition.

\begin{proposition}
    \label{prop_interior_diff}
    For all distinct $\gamma,\sigma,\sigma'\in\Gamma$,
        \[\clInt J(\gamma,\sigma')\setminus\clInt J(\gamma,\sigma)\subset\clInt J(\sigma,\sigma').\]
\end{proposition}

\begin{proof}
    Fix distinct $\gamma,\sigma,\sigma'\in\Gamma$. If $J(\gamma,\sigma')\subset\clInt J(\gamma,\sigma)$, then $\clInt J(\gamma,\sigma')\setminus\clInt J(\gamma,\sigma)=\emptyset$ using Lemma~\ref{lemma_inclusions_deduction} and the result is clear. Let us then assume that
        \[J(\gamma,\sigma')\not\subset\clInt J(\gamma,\sigma)\]
    Let $z\in\clInt J(\gamma,\sigma')\setminus\clInt J(\gamma,\sigma)$. If $z\in J(\gamma,\sigma')$, since $J(\gamma,\sigma)\subset\clInt J(\gamma,\sigma)$, one gets from Proposition~\ref{prop_arc_diff} that
        \[z\in J(\gamma,\sigma')\setminus J(\gamma,\sigma)\subset J(\sigma,\sigma')\subset\clInt J(\sigma,\sigma').\]
    Now, assume $z\notin J(\gamma,\sigma')$, i.e.,
        \[z\in\Int J(\gamma,\sigma')\cap\Ext J(\gamma,\sigma).\]
    Let us distinguish the two cases $J(\gamma,\sigma)\subset\clInt J(\gamma,\sigma')$ and $J(\gamma,\sigma)\not\subset\clInt J(\gamma,\sigma')$.

    \emph{Case 1.} Assume $J(\gamma,\sigma)\subset \clInt J(\gamma,\sigma')$. Then, Corollary~\ref{coro_containment_equiv} gives $J(\sigma,\sigma')\subset\clInt J(\gamma,\sigma')$ and Lemma~\ref{lemma_inclusions_deduction} yields
        \begin{equation}
            \label{eq_inclusion_match}
            \Int J(\sigma,\sigma')\subset\Int J(\gamma,\sigma')\quad\text{and}\quad \Ext(\gamma,\sigma')\subset\Ext J(\sigma,\sigma').
        \end{equation}
    Since $z\in\Ext J(\gamma,\sigma)$, there exists an unbounded continuous path $f:\R_+\to\Ext J(\gamma,\sigma)$ with $f(0)=z$. As $z\in\Int J(\gamma,\sigma')$ and $\Int J(\gamma,\sigma')$ is unbounded, there exists $t_0>0$ such that
        \[f(t_0)\in\partial J(\gamma,\sigma')=J(\gamma,\sigma')\quad\text{and}\quad\forall t\in[0,t_0),~f(t)\in\Int J(\gamma,\sigma').\]
    Observing that $f(t_0)\in J(\gamma,\sigma')\setminus J(\gamma,\sigma)$ from $f(t_0)\in\Ext J(\gamma,\sigma)$, one can apply Lemma~\ref{lemma_local_identification} together with \eqref{eq_inclusion_match} to get an open neighborhood $U$ of $f(t_0)$ such that
        \[U\cap\Int J(\gamma,\sigma')=U\cap\Int J(\sigma,\sigma').\]
    Then, since $f(t)\in \Int J(\gamma,\sigma')$ for $t\in[0,t_0)$, one deduces by continuity of $f$ that there exists $t_0^-\in[0,t_0)$ such that
        \[f(t_0^-)\in\Int J(\sigma,\sigma').\]
    Finally, observing that $f([0,t_0])\subset\R^2\setminus J(\gamma,\sigma')$ avoids $J(\gamma,\sigma)$ and $J(\sigma,\sigma')\setminus J(\gamma,\sigma')\subset J(\gamma,\sigma)$ from Proposition~\ref{prop_arc_diff}, one gets that $f([0,t_0])$ avoids $J(\sigma,\sigma')$, hence $z=f(0)$ and $f(t_0)\in\Int J(\sigma,\sigma')$ must lie in the same component of $\R^2\setminus J(\sigma,\sigma')$, implying $z\in\Int J(\sigma,\sigma')\subset\clInt J(\sigma,\sigma')$.

    \emph{Case 2.} Now, assume $J(\gamma,\sigma)\not\subset \clInt J(\gamma,\sigma')$. Since one has also $J(\gamma,\sigma')\not\subset\clInt J(\gamma,\sigma)$, Lemma~\ref{lemma_inclusions_deduction} together with Corollary~\ref{coro_containment_equiv} gives
        \begin{equation}
            \label{eq_inclusion_match_crossed}
            \Int J(\sigma,\sigma')\subset\Ext J(\gamma,\sigma')\quad\text{and}\quad \Int(\gamma,\sigma')\subset\Ext J(\sigma,\sigma').
        \end{equation}
    Let $f:\R_+\to\R^2$ be an unbounded continuous path with $f(0)=z$. Since $\clInt J(\gamma,\sigma)\cup\clInt J(\gamma,\sigma')$ is bounded, there exists $t_0>0$ such that
        \[f(t_0)\in\partial[\clInt J(\gamma,\sigma)\cup\clInt J(\gamma,\sigma')]\subset\partial\clInt J(\gamma,\sigma)\cup\partial\clInt J(\gamma,\sigma')=J(\gamma,\sigma)\cup J(\gamma,\sigma').\]
    Assume by contradiction that $f(t_0)\notin J(\sigma,\sigma')$. Then, applying Lemma~\ref{lemma_local_identification} with \eqref{eq_inclusion_match_crossed}, there exists an open neighborhood $U$ of $f(t_0)$ such that either
        \[U\setminus J(\gamma,\sigma')=U\cap[\Int J(\gamma,\sigma)\cup\Int J(\gamma,\sigma')]\quad\text{or}\quad U\setminus J(\gamma,\sigma)=U\cap[\Int J(\gamma,\sigma)\cup\Int J(\gamma,\sigma')].\]
    Both cases yield $U\subset\clInt J(\gamma,\sigma)\cup\clInt J(\gamma,\sigma')$ and thus $f(t_0)\notin\partial[\clInt J(\gamma,\sigma)\cup\clInt J(\gamma,\sigma')]$, which is absurd. Therefore, one must have
        \[f(t_0)\in J(\sigma,\sigma').\]
    This shows that any unbounded continuous path started at $z$ must eventually cross $J(\sigma,\sigma')$, i.e., $z\in\Int J(\sigma,\sigma')\subset\clInt J(\sigma,\sigma')$, and thus $\clInt J(\gamma,\sigma')\setminus\clInt J(\gamma,\sigma)\subset\clInt J(\sigma,\sigma')$.
\end{proof}

The following proposition is the last step before proving Theorem~\ref{thm_planar_ordering}. It establishes the sought transitivity and trichotomy properties on $\cdot\in(\cdot,\cdot)_\Gamma$.

\begin{proposition}
    \label{prop_transitive_and_trichotomy}
    The relation $\cdot\in(\cdot,\cdot)_\Gamma$ is \emph{transitive}, that is, for all distinct $\gamma,\sigma',\sigma',\sigma''\in\Gamma$,
        \[\gamma\in(\sigma,\sigma')_\Gamma\setminus(\sigma,\sigma'')_\Gamma\implies\gamma\in (\sigma',\sigma'')_\Gamma.\]
    Furthermore, for all distinct $\sigma,\sigma',\sigma''\in\Gamma$, exactly one of the assertions $\sigma\in(\sigma',\sigma'')_\Gamma$, $\sigma'\in(\sigma,\sigma'')_\Gamma$ and $\sigma''\in(\sigma,\sigma')_\Gamma$ holds.
\end{proposition}

\begin{proof}
    Let $\gamma,\sigma',\sigma',\sigma''\in\Gamma$ be distinct. Assume that
        \[\gamma\in(\sigma,\sigma')_\Gamma\setminus(\sigma,\sigma'')_\Gamma.\]
    Then, using Corollary~\ref{coro_containment_equiv}, one must have $\gamma\in\clInt J(\sigma,\sigma')\cap\clExt J(\sigma,\sigma'')$, so that using Propositions~\ref{prop_arc_diff} and~\ref{prop_interior_diff},
        \[\emptyset\neq\gamma\setminus(\sigma\cup\sigma'\cup\sigma'')\subset\clInt J(\sigma,\sigma')\setminus\clInt J(\sigma,\sigma'')\subset\clInt J(\sigma',\sigma'').\]
    Now, since $\gamma\setminus(\sigma\cup\sigma'\cup\sigma'')$ is contained in $J(\sigma',\gamma)$ and avoids $J(\sigma',\sigma'')$, it follows from Corollary~\ref{coro_containment_equiv} that $J(\sigma',\gamma)\subset\clInt J(\sigma',\sigma'')$, i.e., $\gamma\in (\sigma',\sigma'')_\Gamma$. 
    
    It remains to check the second assertion. Let $\sigma',\sigma',\sigma''\in\Gamma$ be distinct. To show that exactly one of the assertions $\sigma\in(\sigma',\sigma'')_\Gamma$, $\sigma'\in(\sigma,\sigma'')_\Gamma$, and $\sigma''\in(\sigma,\sigma')_\Gamma$ holds, let us first check that at most one of them can hold. Assume $\sigma\in(\sigma',\sigma'')_\Gamma$. Then $J(\sigma,\sigma')\subset\clInt J(\sigma,\sigma'')$, so that $\clInt J(\sigma,\sigma')\subset\clInt J(\sigma,\sigma'')$ using Lemma~\ref{lemma_inclusions_deduction}. Now, as
        \[\partial \clInt J(\sigma,\sigma')=J(\sigma,\sigma')\neq J(\sigma,\sigma')=\partial\clInt J(\sigma,\sigma'')\]
    since $\emptyset\neq\sigma'\setminus(\sigma\cup\sigma'')\subset J(\sigma,\sigma')\setminus J(\sigma,\sigma'')$ from Proposition~\ref{prop_arc_diff}, it implies that
        \[\clInt J(\sigma,\sigma'')\not\subset\clInt J(\sigma,\sigma').\]
    Therefore, using Lemma~\ref{lemma_inclusions_deduction}, it comes $J(\sigma,\sigma'')\not\subset\clInt J(\sigma,\sigma')$, i.e., $\sigma''\notin(\sigma,\sigma')_\Gamma$. This shows
        \[\sigma'\in(\sigma,\sigma'')_\Gamma\implies \sigma''\notin(\sigma,\sigma')_\Gamma.\]
    Then, permuting the roles of $\sigma,\sigma',\sigma''$ and using the symmetry of $(\cdot,\cdot)_\Gamma$ from Corollary~\ref{coro_containment_equiv}, it follows that at most one of the assertions $\sigma\in(\sigma',\sigma'')_\Gamma$, $\sigma'\in(\sigma,\sigma'')_\Gamma$ and $\sigma''\in(\sigma,\sigma')_\Gamma$ can hold. It remains finally to check that at least one of them holds. If $\sigma\in(\sigma',\sigma'')_\Gamma$ or $\sigma'\in(\sigma,\sigma'')_\Gamma$, there is nothing to prove. One can then assume that
        \[\sigma\notin(\sigma',\sigma'')_\Gamma\quad\text{and}\quad \sigma'\notin(\sigma,\sigma'')_\Gamma.\]
    Then since using the symmetry of Corollary~\ref{coro_containment_equiv}, one gets
        \[J(\sigma,\sigma'')\not\subset\clInt J(\sigma',\sigma'')\quad\text{and}\quad J(\sigma,\sigma'')\not\subset\clInt J(\sigma',\sigma'').\]
    Therefore, Lemma~\ref{lemma_inclusions_deduction} gives that $\Int J(\sigma,\sigma'')\subset\Ext J(\sigma',\sigma'')$, and one can write
        \[\Int J(\sigma,\sigma'')\subset \clInt J(\sigma,\sigma'')\setminus\clInt J(\sigma',\sigma'')\subset\clInt J(\sigma,\sigma')\]
    using Proposition~\ref{prop_interior_diff}, so that taking the boundary, $J(\sigma,\sigma'')=\partial\Int J(\sigma,\sigma'')\subset\clInt J(\sigma,\sigma')$. This shows $\sigma''\in(\sigma,\sigma')_\Gamma$ and concludes the proof.
\end{proof}

Putting everything together, we finally prove Theorem~\ref{thm_planar_ordering}.

\begin{proof}[Proof of Theorem~\ref{thm_planar_ordering}]
    Proposition~\ref{prop_arc_diff} gives that for each distinct $\gamma,\gamma'\in\Gamma$, $J(\gamma,\gamma')$ is a Jordan curve. Then, Proposition~\ref{prop_transitive_and_trichotomy} and Corollary~\ref{coro_containment_equiv} ensure that the ternary relation induced on $\Gamma$ by $\cdot\in(\cdot,\cdot)_\Gamma$ setting $(\gamma,\gamma)_\Gamma\coloneq\emptyset$ for all $\gamma\in\Gamma$ satisfies the axioms of betweenness (see \cite{betweenness}). The existence and unicity, up to reversal, of an induced total order follows by construction of betweenness relations. Finally, $\sigma\notin(\gamma,\gamma')_\Gamma\iff J(\gamma,\sigma)\subset\clExt J(\gamma,\gamma')$ for all distinct $\sigma,\gamma,\gamma'\in\Gamma$ is given by Corollary~\ref{coro_containment_equiv}.
\end{proof}

\noindent\textbf{Acknowledgments.} The author warmly thanks his PhD advisors, David Coupier and Viet Chi Tran, for their continuous support throughout this work, as well as Maruf Alam Tarafdar for numerous stimulating discussions that greatly contributed to the development and refinement of the ideas. This research was partly supported by the ANR project \emph{GrHyDy} (ANR-20-CE40-0002), the CEFIPRA project \emph{Directed random networks and their scaling limits} (No.\ 6901), the CNRS RT \emph{MAIAGES} (Action 2179), and by a doctoral fellowship from ENS Paris.

\bibliographystyle{amsalpha}
\bibliography{bibliography}

\end{document}